\documentclass[10pt,a4paper]{article}
\usepackage{amsmath,amsfonts,amsthm,amssymb}
\usepackage{verbatim} 
\usepackage{hyperref}
\usepackage{esint}

\numberwithin{equation}{section}

\newtheorem{lemma}{Lemma}[section]
\newtheorem{remark}[lemma]{Remark}
\newtheorem{cor}[lemma]{Corollary}

\newtheorem{thm}{Theorem}[section]
\newtheorem{prop}[lemma]{Proposition}
\newtheorem{conj}[lemma]{Conjecture}
\newtheorem{notation}[lemma]{Notation}
\newtheorem{definition}[lemma]{Definition}

\providecommand\R{\mathbb{R}}
\providecommand\C{\mathbb{C}}

\providecommand{\grad}{\nabla}

\providecommand{\tO}{\Omega_{\delta}}
\providecommand{\pO}{\partial\Omega}
\providecommand{\rhoO}{\Omega_{\rho/4}^c}
\providecommand{\prhoO}{\partial\Omega_{\rho/4}^c}
\providecommand{\tPsi}{\tilde{\Psi}}
\providecommand{\p}{\partial}

\providecommand{\Div}{\mathrm{div}}
\providecommand{\loc}{\mathrm{loc}}
\DeclareMathOperator*{\dist}{dist}

\def\Xint#1{\mathchoice
{\XXint\displaystyle\textstyle{#1}}%
{\XXint\textstyle\scriptstyle{#1}}%
{\XXint\scriptstyle\scriptscriptstyle{#1}}%
{\XXint\scriptscriptstyle\scriptscriptstyle{#1}}%
\!\int}
\def\XXint#1#2#3{{\setbox0=\hbox{$#1{#2#3}{\int}$ }
\vcenter{\hbox{$#2#3$ }}\kern-.6\wd0}}

\def\dashint{\Xint-}

\begin{document}

\title{Nodal Sets of Steklov Eigenfuntions}
\author{Katar\' ina Bellov\' a  \footnote{Max Planck Institute for Mathematics in the Sciences, Leipzig, Germany, email: bellova@mis.mpg.de} \and Fang-Hua Lin
\footnote{Courant Institute of Mathematical Sciences, New York University, email: linf@cims.nyu.edu } }



\maketitle

\begin{abstract}
We study the nodal set of the Steklov eigenfunctions on the boundary of a smooth bounded domain in $\R^n$ 
-- the eigenfunctions of the Dirichlet-to-Neumann map. 
Under the assumption that the domain $\Omega$ is $C^2$, we prove a doubling property for the eigenfunction $u$.
We estimate the Hausdorff $\mathcal H^{n-2}$-measure
of the nodal set of $u|_{\pO}$ in terms of the eigenvalue $\lambda$ as 
$\lambda$ grows to infinity.
In case that the domain $\Omega$ is analytic, we prove a polynomial bound 
O($\lambda^6$). 
Our arguments, which make heavy use of Almgren's frequency functions,
are built on the previous works 
[Garofalo and Lin, CPAM {\bf 40} (1987), no.~3;
Lin, CPAM {\bf 42} (1989), no.~6].
\end{abstract}


\section{Introduction}

In this paper, we study the nodal set of the Steklov eigenfunctions on the boundary of a smooth bounded domain in $\R^n$ 
-- the eigenfunctions of the Dirichlet-to-Neumann map $\Lambda$. 
For a bounded Lipschitz domain $\Omega\subset\R^n$,
this map $\Lambda$
associates to each function $u$ defined
on the boundary $\partial \Omega$, the normal derivative of the harmonic function on $\Omega$
with boundary data $u$.
More generally, one can consider an $n$-dimensional smooth Riemannian manifold $(M,g)$
instead of $\Omega$, and replace the Laplacian 
by the Laplace-Beltrami operator $\Delta_g$.
Our methods, which build on 
the papers \cite{lin_garofalo_2,lin_nodal}, 
can be used also for solutions to more general elliptic equations and 
general boundary conditions involving oblique derivatives. 

Steklov eigenfunctions were introduced by Steklov \cite{steklov} in 1902 for bounded domains in the plane.
They represent the steady state temperature distribution on $\Omega$ such 
that the heat flux on the boundary is proportional to the temperature. 
The problem can also be interpreted as vibration of a free membrane with the mass uniformly distributed on the boundary.
Note that in this case, the eigenfunction's nodal set would represents 
the stationary points on the boundary. The studies of Steklov-type
eigenvalue problems are related to several important problems in
differential geometry, see for examples, 
\cite{escobar_yamabe,escobar_first_ev,marques_yamabe,fraser_schoen,fraser_schoen_2013}.
They are also closely connected with some classical geometric inequalities
and Sobolev trace inequalities, see 
\cite{escobar_sobolev_ineq,escobar_first_ev,li_zhu_sobolev_ineq,li-zhu,sharp_constants_survey,payne,weinstock,yau}.
It is also well-known that the Dirichlet to Neumann map is an
essential tool for studies of many inverse problems. See for examples
\cite{ammari,polterovich_sloshing,calderon,partial_n_d_map,uhlmann_2007,sylvester_uhlmann_ann,fox_kuttler}.
Generalizations of Dirichlet to Neumann
maps are also related to elliptic operators of fractional order \cite{caffarelli_silvestre,caffarelli_obstacle_for_frac_lapl,chang-gonzalez,graham-zworski}.
Though the first Steklov eigenfunction is 
particularly related to
geometric applications and extreme inequalities, higher Steklov
eigenfunctions and their distributions have also been studied by
various authors, see 
\cite{bandle,dittmar,hersch_payne,fraser_schoen_2013,pleijel,sandgren} or 
\cite[Chapter 17, Section 5]{hormander}. 
The present paper is
devoted to a general study of nodal sets of Steklov eigenfunctions.
It is our first attempt to understand nodal sets of solutions of nonlocal
elliptic operators or pseudodifferential operators. 


The paper is organized as follows. The remainder of the introduction summarizes the main results
and fixes some basic notation. Section \ref{sec_frequency} recalls those results 
from \cite{lin_garofalo,lin_garofalo_2} about the frequency function,
which we will need in our paper.
In Section \ref{sec_doubling} we prove a doubling condition for Steklov eigenfunctions
on a $C^2$-domain  $\Omega$, which will serve as a cornerstone for the nodal set estimate in the analytic setting.
In Section \ref{chap_analytic} we prove an explicit estimate for the nodal set of Steklov eigenfunctions in the case that $\Omega$ has analytic boundary.

\subsection{Problem setting and main results} \label{problem_setting}
Let $\Omega\subset\R^n$ be a Lipschitz domain. 
The Dirichlet-to-Neumann operator $\Lambda: H^{1/2}(\pO) \to H^{-1/2}(\pO)$
is defined as follows. For $f\in H^{1/2}(\pO)$, we solve the Laplace equation
\begin{equation*}
\begin{aligned}
  \Delta u = 0 \quad &\text{ in } \Omega, \\
  u = f \quad &\text{ on } \pO.
\end{aligned}
\end{equation*}
This gives a solution $u\in H^{1}({\Omega})$, and we set $(\Lambda f)$ to be the trace of $\frac{\partial u}{\partial \nu}$ on $\pO$,
where $\nu$ is the exterior unit normal.
We obtain a bounded self-adjoint operator 
from $H^{1/2}(\pO)$ to $H^{-1/2}(\pO)$.
It has a discrete spectrum $\{\lambda_j \}_{j=0}^{\infty}$, $0=\lambda_0<\lambda_1\leq\lambda_2\leq\lambda_3\leq...$,
$\lim_{j\to\infty}\lambda_j = \infty$. The eigenfunctions of $\Lambda$ 
(called Steklov eigenfunctions) corresponding to eigenvalue $\lambda$ 
can be identified with the trace on $\pO$ of their harmonic extensions
to $\Omega$, which satisfy
\begin{equation} \label{main}
\begin{aligned}
\Delta u &= 0  &\text{ in } \Omega, \\
\frac{\partial u}{\partial \nu} &= \lambda u &\text{ on } \partial\Omega.
\end{aligned}
\end{equation}

The main goal of this paper 
is to estimate the size (the Hausdorff $\mathcal H^{n-2}$-measure)
of the nodal set of $u|_{\pO}$ in terms of $\lambda$ as $\lambda$ grows to infinity, provided $\Omega$ is fixed.
In Section \ref{chap_analytic} we prove the following bound in case that $\Omega$ has analytic boundary:

\begin{thm}\label{thm_steklov_analytic}
 Let $\Omega \subset \R^n$ be an analytic domain. Then there exists a constant $C$ depending
only on $\Omega$ and $n$ such that for any $\lambda>0$ and $u$ 
which is a (classical) solution to \eqref{main} there holds
\begin{equation} \label{steklov_analytic}
 \mathcal H^{n-2} (\{x\in\pO: u(x)=0\}) \leq C \lambda^6.
\end{equation}
\end{thm}

\noindent
The scaling $\lambda^6$ in \eqref{steklov_analytic} is not optimal. Actually, our proof 
gives approximately $\lambda^{5.6}$, but we believe that the optimal scaling is $\lambda$.
However, even this polynomial bound is valuable in a problem like this -- the main difficulty
in the estimate is to avoid an exponential bound $e^{C \lambda}$.

As in \cite{lin_nodal} and \cite{han_lin_slices}, we use
a doubling condition, i.e.~a control of the $L^2$-norm of $u$ on a ball $B_{2r}(x)$ by the $L^2$-norm of $u$
on a smaller ball $B_r(x)$, as the crucial tool
to estimate the nodal set. We prove a doubling condition in the following form in Section \ref{sec_doubling}:
\begin{thm} \label{thm_doubling_bdry}
Let $\Omega \subset \R^n$ be a $C^2$ domain.
 Then there exist constants $r_0, C>0$ depending only on $\Omega$ and $n$ such that 
for any $r\leq r_0/\lambda$, $x\in \pO$, and $u$ 
which is a (classical) solution to \eqref{main}, there holds
\begin{equation} \label{doubling_bdry}  
\int_{B(x,2r) \cap \pO}  u^2 \leq  2^{C\lambda^5} \int_{ B(x,r) \cap \pO}  u^2.
\end{equation}
\end{thm}

To put our work into context, note that our
problem is similar in nature to the classical question of estimating the size
of nodal sets of eigenfunctions of the Laplace operator in a compact manifold. 
The following conjecture was proposed by Yau in \cite{yau}: 
\begin{conj} \label{conj_yau}
Suppose
$(M^n, g)$ is a smooth $n$-dimensional connected and compact Riemannian
manifold without boundary. Consider an eigenfunction $u$
corresponding to the eigenvalue $\lambda$, i.e.~,
\begin{equation*}
 \Delta_g u + \lambda u = 0 \quad \text{ on } M.
\end{equation*}
Then there holds
\begin{equation*}
c_1 \sqrt{\lambda} \leq \mathcal{H}^{n-1}(\{x \in M; u(x) = 0 \}) \leq c_2 \sqrt{\lambda},
\end{equation*}
where $c_1$ and $c_2$ are positive constants depending only on $(M, g)$.
\end{conj}
\noindent This conjecture was proved in case that $(M^n, g)$ is analytic by Donnelly and Fefferman in \cite{donnelly_fefferman}.
It is still open whether Conjecture \ref{conj_yau}
holds if $(M^n, g)$ is only smooth. The known results for the smooth case are far from optimal,
the upper bound remaining exponential (see \cite{hardt_simon}). One may 
ask similar questions for Steklov eigenfunctions or general solutions of 
nonlocal elliptic operators.

Another related problem has been studied for the Neumann eigenfunctions on 
a piecewise analytic plane domain $\Omega\subset\R^2$
in \cite{zelditch_nodal_lines}. 
This paper is concerned about the asymptotics of the number of 
nodal points of the eigenfunctions on the boundary $\partial \Omega$, 
as the eigenvalue $\lambda$ increases to infinity. 
It proves that this number is bounded above by $C_{\Omega}\lambda$.

\subsection{Notation}
Throughout this paper, 
$B_r=B(0,r)$ will denote the open ball in $\R^n$ with center $0$ and radius $r$,
and $B(x,r)$ the open ball in $\R^n$ with center $x$ and radius $r$. We use
$B^{n-1}(x,r)$ for ball in $\R^{n-1}$ and $B^{\C^k}(z,r)$
for ball in $\C^k$ ($k=n,n-1$).

We denote the coordinates of a vector $b\in \R^n$ by $b^i$, $i=1,2,\dots,n$.
We abbreviate the partial derivatives as $\frac{\p w}{\p x_i} = w_{x_i}$,
$\frac{\p^2 w}{\p x_i \p x_j} = w_{x_i,x_j}$.

We use the letter $C$ as a constant (usually depending only on $n$ and $\Omega$),
which can change from line to line. 

For a given $\Omega$, there are only finitely many eigenvalues $\lambda$ 
of the Dirichlet-to-Neumann map which
are less than $1$. Hence, when we are proving upper bounds in the form of $C \lambda^k$ with specific $k$
and $C$ depending on $\Omega$,
without loss of generality we can assume that the eigenvalue
$\lambda$ is larger or equal to $1$, and we often do so without pointing it out.

\section{The Frequency Function} \label{sec_frequency}
In this section we review the theory developed in \cite{lin_garofalo} and \cite{lin_garofalo_2}
about the frequency functions for both harmonic functions and solutions to general elliptic equations.
We will use the frequency function as the main tool to derive the doubling condition,
as e.g.~in \cite{lin_garofalo,lin_garofalo_2,lin_nodal,han_harmonic}.

\subsection{Frequency function and doubling condition for harmonic functions}

\begin{definition}
For a harmonic function $u$ on ball $B_1$ and $r<1$, the frequency
$N(r)$ is defined as
\begin{equation} \label{frequency}
N(r)= \frac{r D(r)}{H(r)},
\end{equation}
where
\begin{align*}
D(r)&= \int_{B_r} |\grad u|^2 dx, \\
H(r)&= \int_{\partial B_r} u^2 d \sigma,
\end{align*}
where $ \sigma$ is the surface measure.
For a harmonic function on a ball $B(a,r)$, $N(a,r)$, $D(a,r)$ and $H(a,r)$ are defined analogously.
\end{definition}
The frequency is a way how to measure the growth of a harmonic function. If $u$ is a homogeneous
harmonic polynomial, its frequency is exactly its degree. See \cite{han_harmonic} or \cite{book}
for more examples. Let us list some important properties. We refer to the survey paper \cite{han_harmonic}
for proofs, although they are known much longer and sketches of the proofs can be found
e.g.~in \cite{lin_garofalo} or \cite{lin_nodal}. The following monotonicity property of the frequency function
is attributed to F.~J.~Almgren, Jr., \cite{almgren}.

\begin{prop}[\cite{han_harmonic}, Theorem 1.2] \label{monotonicity_harm}
 Let $u$ be a harmonic function in $B_1$. Then $N(r)$ is a nondecreasing function of $r\in (0,1)$.
\end{prop}

\begin{prop}[\cite{han_harmonic}, Corollary 1.4] \label{prop_H_prime}
Let $u$ be a harmonic function in $B_1$. For any $r\in(0,1)$, there holds
\begin{equation}\label{H_prime}
 \frac{d}{dr} \left( \log \frac{H(r)}{r^{n-1}} \right) = 2\frac{N(r)}{r}.
\end{equation}
Integrating \eqref{H_prime}, we obtain that for any $0<r_1<r_2<1$, there holds
\begin{equation} \label{H_eq}
\frac{H(r_2) }{r_2^{n-1}} = \frac{H(r_1) }{r_1^{n-1}} \exp \left(2\int_{r_1}^{r_2} \frac{N(r)}{r}\right).
\end{equation}
Using the monotonicity of $N$ (Proposition \ref{monotonicity_harm}), it follows that
\begin{equation} \label{H_ineq}
\frac{H(r_2) }{r_2^{n-1}} \leq  \left(\frac{r_2}{r_1}\right)^{2N(r_2)} \frac{H(r_1) }{r_1^{n-1}} .
\end{equation}
\end{prop}

\begin{cor} \label{cor_reverse_doubling}
Let $u$ be a harmonic function in $B_1$.
Then the function 
\begin{equation} \label{reverse_doubling}
r \mapsto \dashint_{\partial B_r} u^2
\end{equation}
is increasing with respect to $r$, $r\in (0,1)$, and
\begin{equation} \label{reverse_doubling2}
 \dashint_{B_r} u^2 \leq \dashint_{\p B_r} u^2.
\end{equation}

\end{cor}

\begin{proof}
 It follows directly from \eqref{H_prime} that the function \eqref{reverse_doubling} is increasing.
By integration, we get \eqref{reverse_doubling2}.
\end{proof}

The following result, if we take $\eta=1/2$, is called the {\it doubling condition}. It is a counterpart of
Corollary \ref{cor_reverse_doubling}.
\begin{cor} \label{cor_doubling}
 Let $u$ be a harmonic function in $B_1$. For any $R,\eta \in(0, 1)$, there holds
\begin{align}
 \dashint_{\p B_{R}} u^2 \leq \eta^{-2N(R)}    \dashint_{\p B_{\eta R}} u^2,   \label{H_doubling_harm} \\
 \dashint_{ B_{R}} u^2 \leq \eta^{-2N(R)} \dashint_{ B_{\eta R}} u^2.          \label{doubling_harm}
\end{align}
\end{cor}

\begin{proof}
 Taking $r_1=\eta R, r_2=R$ in \eqref{H_ineq}, we obtain
\eqref{H_doubling_harm}. Integrating \eqref{H_doubling_harm} from $0$ to $R$, 
and using the monotonicity of $N$ (Proposition \ref{monotonicity_harm}),
we obtain \eqref{doubling_harm}.
\end{proof}

Next, we show that not only having a bound on the frequency implies a doubling condition
(Corollary \ref{cor_doubling}), but also knowing a doubling condition to be true implies
a bound on the frequency.

\begin{lemma} \label{lemma_freq_from_doubling}
 Let $u$ be a harmonic function in $B_r$, $r>0$. Let $0<\alpha<\theta<1$, and assume
\begin{equation} \label{assume}
\dashint_{B_{\alpha r}} u^2 \geq \kappa \dashint_{B_r} u^2
\end{equation}
for some $\kappa>0$.
Then
\begin{equation} \label{freq_bound}
N(\alpha r) \leq \frac{- \log \left( {\kappa} (1-\theta^n) \right)}{2 \log \left(\theta/ \alpha \right)}.
\end{equation}
In particular, for any $\beta<\alpha$, there holds
\begin{equation}\label{doubling_ab}
\dashint_{B_{\beta r}} u^2 
\geq  
\left( {\kappa (1-\theta^n)}  \right)^{\frac{\log (\alpha/\beta)}{ \log (\theta/\alpha)}} \dashint_{B_{\alpha r}} u^2.
\end{equation}
\end{lemma}

\begin{proof}
 Using Corollary \ref{cor_reverse_doubling} and  the assumption \eqref{assume}, we obtain
\begin{align*}
 \dashint_{\p B_{\alpha r}} u^2 &\geq \dashint_{B_{\alpha r}} u^2 \geq \kappa \dashint_{ B_{ r}} u^2
\geq \kappa \frac{1}{\omega_n r^n} \int_{\theta r}^r \int_{\p B_{\rho}} u^2 d\sigma d\rho   \\
&\geq\kappa \frac{1}{\omega_n r^n} \int_{\theta r}^r \left(\frac{\rho}{\theta r} \right)^{n-1} d\rho \int_{\p B_{\theta r}} u^2 
= \kappa (1-\theta^n) \dashint_{\p B_{\theta r}} u^2.
\end{align*}
Using \eqref{H_eq} and the monotonicity of $N$, we further obtain
$$
-\log (\kappa (1-\theta^n)) \geq \log \left(\frac{H(\theta r)/(\theta r)^{n-1}}{ H(\alpha r)/(\alpha r)^{n-1}}   \right) =
2\int_{\alpha r}^{\theta r} \frac{N(\rho)}{\rho} \geq 2N(\alpha r) \log (\theta/\alpha),
$$
which is \eqref{freq_bound}. The inequality \eqref{doubling_ab} follows from this bound on $N(\alpha r)$ and
the doubling condition \eqref{doubling_harm}:
\begin{align*}
\dashint_{B_{\beta r}} u^2 
\geq  
\left(\frac{\beta}{\alpha} \right)^{2N(\alpha r)} \dashint_{B_{\alpha r}} u^2
\geq
\left(\frac{\beta}{\alpha} \right)^{ \frac{- \log \left( {\kappa} (1-\theta^n) \right)}{ \log \left(\theta/ \alpha \right)}}
\dashint_{B_{\alpha r}} u^2
\\
=
\left( {\kappa (1-\theta^n)}  \right)^{\frac{\log (\alpha/\beta)}{ \log (\theta/\alpha)}} \dashint_{B_{\alpha r}} u^2.
\end{align*}
\end{proof}

Next, we recall a result estimating the frequency in a given point by the frequency
in a different point. This is an important property whose adaptation we will use to
obtain global estimates.

\begin{prop}[\cite{han_harmonic}, Theorem 1.6] \label{chain_frequency}
 Let $u$ be a harmonic function in $B_1$. For any $R\in (0,1)$, there  exists a constant
$N_0=N_0(R)\ll 1$ such that the following holds. If $N(0,1)\leq N_0$, then $u$ does not vanish in $B_R$.
If $N(0,1)\geq N_0$, then there holds
$$
N\left(p, \frac{1}{2}(1-R)  \right) \leq CN(0,1), \quad \text{for any } p\in B_R,
$$
where $C$ is a positive constant depending only on $n$ and $R$. In particular, 
the vanishing order of $u$ at any point in $B_R$ never exceeds $c(n, R)N(0,1)$.
\end{prop}

\subsection{Frequency and doubling condition for solutions of general elliptic equations} \label{sec_frequency_gen}
\medskip

We will need a generalization of the frequency \eqref{frequency} for more general
elliptic equations. Let us recall the theory developed in \cite{lin_garofalo} and \cite{lin_garofalo_2}.

In the unit ball $B_1 \subset \R^n$, 
consider the equation
\begin{equation} \label{equation}
Lw= \Div (A(x)\grad w) + b(x)\cdot \grad w +c(x) w =0,
\end{equation}
where $A(x) = (a^{ij}(x))_{i,j=1}^n$ is a real symmetric matrix-valued function on $B_1$, 
$b(x)$ is a vector valued and $c(x)$ is a scalar function on $B_1$.
We assume

(i) there exists $\alpha\in(0,1)$ such that, for every $x\in B_1$ and  $\xi\in\R^n$,
\begin{equation} \label{ellipticity}
 \alpha|\xi|^2 \leq \sum_{i,j=1}^n a^{ij}(x) \xi_i \xi_j;
\end{equation}

(ii) there exists $\gamma>0$ such that, for every $x,y\in B_1$,
\begin{equation} \label{Lipschitz}
 |a^{ij}(x) - a^{ij}(y)|\leq \gamma |x-y|, \hskip1cm i,j=1,\dots, n;
\end{equation}

(iii) there exists $K>0$ such that
\begin{equation}	 \label{L-infinity}
 \sum_{i,j}|| a^{ij} ||_{L^{\infty}(B_1)} + \sum_j || b^j ||_{L^{\infty}(B_1)}+|| c ||_{L^{\infty}(B_1)} \leq K.
 \end{equation}

Since we will use the frequency only in points $x$ where $A(x)=I$,
i.~e.~where the principal part of the equation is Laplacian, and the
exposition becomes significantly simpler in this setting, we will
limit ourselves to defining the frequency only in these points.
The general definition (and properties, which hold in the same way)
can be found in \cite{lin_garofalo,lin_garofalo_2}.

\begin{notation}
Let $w\in W_{\loc}^{1,2}(B_1)$ be a solution of \eqref{equation} in $B_1$, $A(0)=I$, and $0<r<1$.
Denote
\begin{gather*}
H(r) = \int_{\p B_r}  w^2 \,d \sigma  \\
D(r)= \int_{B_r}  |\grad\, w|^2 \, dx, \\
I(r)= \int_{B_r}(|\grad \, w|^2 + ub\cdot \grad\, w +c w^2 ) \, dx, \\
N(r)=\frac{r I(r)}{H(r)},
\end{gather*}
the last quantity being defined only if $H(r)>0$.
Analogously we define these quantities not only for solutions $w$ on the unit ball $B_1=B(0,1)$,
but also for solutions on any ball $B(x_0,r_0)$ with center $x_0$ such that $A(x_0)=I$ and radius $r_0$.
Then $H(r)$, $D(r)$, $I(r)$ and $N(r)$ also depend on $w$ and $x_0$, and if the function
and the center are not clear from the context, we will denote them $H_w(x_0,r)$, $D_w(x_0,r)$, $I_w(x_0,r)$, $N_w(x_0,r)$
(we will skip either $x_0$ or $w$ if just one is not clear from the context).
\end{notation}

Note that unlike in the harmonic case, it is not clear anymore where $N(r)$ is defined and whether $N(r)>0$: we have $D(r)>0, H(r)\geq 0$,
but the sign of $I(r)$ is not obvious.
However, it can be shown that at least on some interval, $N(r)$ is defined and $N(r)/r\geq -C$ for some constant $C$.

\begin{lemma} \label{H_positive}
  Let $w\in W_{\loc}^{1,2}(B_1)$ be a nonzero solution to \eqref{equation} in $B_1$, $A(0)=I$. Then there exist constants $r_0,C>0$,
depending on $n,\alpha,\gamma, K$, such that for any $r\in (0,r_0)$,
\begin{align}
D(r) & \leq 2I(r) +C H(r), \label{I_D_inequality} \\
\nonumber D(r) & \geq  \frac{1}{2} I(r) -C H(r), \\
\nonumber H(r) &>0.
\end{align}
\end{lemma}
 
\begin{proof}
 The lemma can be found in \cite{book} as Lemma 3.2.3 and Corollary 3.2.5. In \cite{lin_garofalo_2}, the last
inequality is Lemma 2.2. and the previous two easily follow from the estimates done in the proof of this lemma.
\end{proof}

\begin{cor} \label{N_almost_positive}
 Let $w\in W_{\loc}^{1,2}(B_1)$ be a nonzero solution to \eqref{equation} in $B_1$, $A(0)=I$. Then there exist constants $r_0,C>0$,
depending on $n,\alpha,\gamma, K$, such that for any $r\in (0,r_0)$, $N(r)$ is defined and
\begin{equation*}
\frac{N(r)}{r} \geq -C. 
\end{equation*}
\end{cor}

\begin{proof}
Take $r_0$ from Lemma \ref{H_positive}. Then for $r\in (0,r_0)$, $H(r)$ is positive, so $N(r)$ is defined and
by \eqref{I_D_inequality}, we obtain
\begin{equation*}
 \frac{N(r)}{r} = \frac{I(r)}{H(r)} \geq \frac{D(r)}{2H(r)} - C \geq -C.
\end{equation*}
\end{proof}

Corresponding to the monotonicity of $N$ in the harmonic case, there holds the following bound in the general elliptic setting.
\begin{thm} \label{N_L_infty}
 Let $w\in W_{\loc}^{1,2}(B_1)$ be a nonzero solution to \eqref{equation} in $B_1$, $A(0)=I$. Then there exist constants $r_0,c_1,c_2>0$,
depending on $n,\alpha,\gamma, K$, such that 
\begin{equation} \label{N_L_infty_bound}
N(R_1) 
\leq c_1+ c_2 N(R_2) \hskip1cm \text{ for any } 0<R_1 < R_2 \leq r_0. 
\end{equation}
\end{thm}
\begin{proof}
This theorem follows
from Theorem $2.1$ in \cite{lin_garofalo_2}, using our bounds \eqref{L-infinity}
instead of their more general assumptions on $b$ and $c$, and checking the $L^{\infty}$
bound on $N$ in the proof. It can also be found in a more similar form in \cite{book}
as Theorem $3.2.1$. The only difference in our formulation is that we state the estimate for
any $R_2 \leq r_0$, not just $R_2=r_0$. However, going through the proofs of the theorems
in \cite{lin_garofalo_2} or \cite{book}, one can check that the estimate is true
with our formulation without any changes to the proofs. Another way how to verify
inequality \eqref{N_L_infty_bound} for $R_2<r_0$ if we know that it holds for 
$R_2=r_0$ is to consider a solution $w_{r_0/R_2}=w(\frac{R_2 x}{r_0} )$ of a scaled equation
(which satisfies the assumptions \eqref{ellipticity}, \eqref{Lipschitz} and \eqref{L-infinity}
with the same constants as the original equation),
use \eqref{N_L_infty_bound} for $w_{r_0/R_2}$ and $\widetilde{R_2}=r_0$, and then use the scaling of
$N$ explained in the next section to deduce \eqref{N_L_infty_bound} for $w$ and the original $R_2<r_0$.
\end{proof}

Next, we will recall results for solutions of general elliptic equations corresponding
to Proposition \ref{prop_H_prime}, Corollary \ref{cor_reverse_doubling} and Corollary \ref{cor_doubling}.

\begin{prop} \label{prop_H_prime_gen}
 Let $w\in W_{\loc}^{1,2}(B_1)$ be a nonzero solution to \eqref{equation} in $B_1$, $A(0)=I$. Then  
\begin{equation}\label{H_prime_gen}
 \frac{d}{dr} \left( \log \frac{H(r)}{r^{n-1}} \right) = O(1)+ 2\frac{N(r)}{r},
\end{equation}
where $O(1)$ denotes a function bounded by a constant depending on $n, \alpha, \gamma$ and
the $L^{\infty}$-bound on the leading order coefficients $A$.
Integrating \eqref{H_prime_gen}, we obtain that for any $0<R_1<R_2<1$, there holds
\begin{equation} \label{H_eq_gen}
\frac{H(R_2) }{R_2^{n-1}} = \frac{H(R_1) }{R_1^{n-1}} \exp   \left( O(1)(R_2-R_1) + 2 \int_{R_1}^{R_2}  \frac{N(r)}{r}  \right).
\end{equation}
Using Theorem \ref{N_L_infty} (the bound for $N$), it follows that if $R_2 \leq r_0$ (where $r_0$ comes from Theorem \ref{N_L_infty}
and depends on $n,\alpha, \gamma, K$), then
\begin{equation} \label{H_ineq_gen}
\frac{H(R_2) }{R_2^{n-1}} \leq  C^{(R_2-R_1)} \left(\frac{R_2}{R_1}\right)^{c_1+ c_2 N(R_2)} \frac{H(R_1) }{R_1^{n-1}}, 
\end{equation}
where $C,c_1,c_2$ depend on  $n,\alpha, \gamma, K$.
\end{prop}

\begin{proof}
Equation \eqref{H_prime_gen} can be found in \cite{lin_garofalo_2} as equation $(2.16)$. The fact that
$O(1)$ does not depend on the coefficients $b$ and $c$ can be verified by going back to the proof of this
equation. The $O(1)$ is the same as in \cite{lin_garofalo} for equations with no lower order terms.
\end{proof}

\begin{cor} \label{cor_reverse_doubling_gen}
 Let $w\in W_{\loc}^{1,2}(B_1)$ be a nonzero solution to \eqref{equation} in $B_1$, $A(0)=I$. 
Then there exist constants $r_0,C>0$ depending only on $n, \alpha, \gamma, K$ 
such that 
\begin{equation} \label{reverse_doubling_gen}
\dashint_{\partial B_s} w^2 \leq C \dashint_{\partial B_r} w^2 \ \ \  \text{ for any } 0<s<r<r_0
\end{equation}
and
\begin{equation} \label{reverse_doubling2_gen}
 \dashint_{B_r} w^2 \leq  
C \dashint_{\p B_r} w^2 \ \ \  \text{ for } r\in (0,r_0).
\end{equation}
\end{cor}

\begin{proof}
Using Corollary \ref{N_almost_positive}, i.e.~$N(r)/r\geq -C$,
and \eqref{H_prime_gen}, we obtain 
\begin{equation*}
 \frac{d}{dr} \left( \log \frac{H(r)}{r^{n-1}} \right) = O(1)+ 2\frac{N(r)}{r} \geq -C
\end{equation*}
for some constants $r_0,C$ and $r<r_0$.
Hence the function $e^{Cr}H(r)/r^{n-1} $ is increasing
and \eqref{reverse_doubling_gen} follows.
By integration, we get \eqref{reverse_doubling2_gen}.
\end{proof}


\begin{thm} \label{thm_doubling_gen}
Let $w\in W_{\loc}^{1,2}(B_1)$ be a nonzero solution to \eqref{equation} in $B_1$, $A(0)=I$. Then there exist constants $r_0,C,c_1,c_2>0$,
depending on $n,\alpha,\gamma, K$, such that for any $0<R_1<R_2<r_0$,
\begin{align}
 \dashint_{\p B_{R_2}} w^2 \leq C \left( \frac{R_2}{R_1} \right)^{c_1+c_2N(R_2)}    \dashint_{\p B_{ R_1}} w^2,   \label{H_doubling_gen} \\
 \dashint_{ B_{R_2}} w^2 \leq C \left( \frac{R_2}{R_1} \right)^{c_1+c_2N(R_2)} \dashint_{ B_{ R_1}} w^2.          \label{doubling_gen}
\end{align}
\end{thm}

\begin{proof}
 This can be found in a slightly different form in \cite{lin_garofalo_2} as Theorem 1.2, or in
\cite{book} as Theorem 3.2.7. It also follows from \eqref{H_ineq_gen}. 
\end{proof}

Finally, let us show the equivalent of Lemma \ref{lemma_freq_from_doubling}:
knowing a doubling condition to be true implies a bound on the frequency.

\begin{lemma} \label{lemma_freq_from_doubling_gen}
Let $w\in W_{\loc}^{1,2}(B_1)$ be a nonzero solution to \eqref{equation} in $B_1$, $A(0)=I$. 
Assume
\begin{equation} \label{assume_gen}
\dashint_{B_{\zeta r}} w^2 \geq \kappa \dashint_{B_r} w^2
\end{equation}
for some $\kappa, \zeta \in (0,1)$ and $r\in (0,r_0]$, where $r_0$ depends on $n, \alpha, \gamma, K$ 
and is chosen so that the previous properties in this subsection hold.
Then there exists a constant $C_{\zeta}>0$ depending on $n$, $\alpha$, $\gamma$, $K$ and $\zeta$
such that
\begin{equation} \label{freq_bound_gen}
N(\zeta r) \leq C_{\zeta} (1- \log \kappa).
\end{equation}
In particular, for any $\beta< \zeta $, there exist constants $C_1$, $C_2$  depending on $n$, $\alpha$, $\gamma$, $K$ and $\zeta$, $\beta$ such that
\begin{equation}\label{doubling_ab_gen}
\dashint_{B_{\beta r}} w^2 
\geq  
\frac{1}{C_1} \kappa^{C_2}
\dashint_{B_{\zeta r}} w^2.
\end{equation}
\end{lemma}

\begin{proof}
We copy the proof of Lemma \ref{lemma_freq_from_doubling}. Choose $\theta$ such that $\zeta <\theta <1$, e.g.~$\theta=\frac{1+\zeta}{2}$.
 Using Corollary \ref{cor_reverse_doubling_gen} and  the assumption \eqref{assume_gen}, we obtain
\begin{align*}
 \dashint_{\p B_{\zeta r}} w^2 &\geq \frac{1}{C} \dashint_{B_{\zeta r}} w^2 \geq \frac{1}{C} \kappa \dashint_{ B_{ r}} w^2
\geq \frac{1}{C} \kappa \frac{1}{\omega_n r^n} \int_{\theta r}^r \int_{\p B_{\rho}} w^2 d\sigma d\rho  \\
&\geq \frac{1}{C} \kappa \frac{1}{\omega_n r^n} \int_{\theta r}^r \left(\frac{\rho}{\theta r} \right)^{n-1} d\rho \int_{\p B_{\theta r}} w^2 
= \frac{1}{C} \kappa (1-\theta^n) \dashint_{\p B_{\theta r}} w^2.
\end{align*}
Using \eqref{H_eq_gen},  
and Theorem \ref{N_L_infty}, 
we further obtain
\begin{align*}
C -\log ( \kappa (1-\theta^n)) &\geq \log \left(\frac{H(\theta r)/(\theta r)^{n-1}}{ H(\zeta r)/(\zeta r)^{n-1}}   \right) \\
&= O(1)(\theta -\zeta)r +
2\int_{\zeta r}^{\theta r} \frac{N(\rho)}{\rho} \\
&\geq -C + \frac{1}{C} N(\zeta r) \log (\theta/\zeta),
\end{align*}
and \eqref{freq_bound_gen} follows. The inequality \eqref{doubling_ab_gen} follows from this bound on $N(\zeta r)$ and
the doubling condition \eqref{doubling_gen}:
\begin{align*}
\dashint_{B_{\beta r}} w^2 
\geq  
\frac{1}{C} \left(\frac{\beta}{\zeta} \right)^{c_1+c_2 N(\zeta r)} \dashint_{B_{\zeta r}} w^2
\geq
\frac{1}{C}\left(\frac{\beta}{\zeta} \right)^{\widetilde{C_1}-\widetilde{C_2} \log \kappa}
\dashint_{B_{\zeta r}} w^2
\\
=
\frac{1}{C_1}\kappa^{C_2}
\dashint_{B_{\zeta r}} w^2.
\end{align*}
\end{proof}


\section{Doubling Condition for Steklov Eigenfunctions} \label{sec_doubling}
In this section, $\Omega\subset \R^n$ will  denote a $C^2$ domain, and $u$ will be a Steklov eigenfunction corresponding
to eigenvalue $\lambda$, harmonically extended to $\Omega$, i.e~satisfying \eqref{main}:
\begin{equation*} 
\begin{aligned}
\Delta u &= 0  &\text{ in } \Omega, \\
\frac{\partial u}{\partial \nu} &= \lambda u &\text{ on } \partial\Omega.
\end{aligned}
\end{equation*}

In this section we prove Theorem \ref{thm_doubling_bdry}, i.~e.~the doubling condition -- controlling the $L^2$-norm of $u$ on a ball $B_{2r}(x)$ by the $L^2$-norm of $u$
on a smaller ball $B_r(x)$.

\subsection{Reduction to Neumann boundary condition and reflection across boundary} \label{reflection}
\begin{notation} \label{delta}
Since $\partial\Omega$ is compact and $C^2$, it has bounded curvature
and there exists $\delta>0$ such that the map 
$$
(y,t) \mapsto y+t\nu(y)
$$ 
is one-to-one from $\partial\Omega \times (-\delta,\delta)$ onto $\delta$-neighborhood of $\partial\Omega$. 
Fix this $\delta$ (depending on $\Omega$) for the rest of this section.
For any $\rho\leq \delta$, denote 
\begin{align*}
\{y+t\nu(y)\  |\  y\in\partial\Omega, t\in(-\rho,0)\} &
= \{x\in\R^n| \dist(x,\partial\Omega) < \rho\} \cap\Omega
=\Omega_{\rho}, \\
\{y+t\nu(y)\  |\  y\in\partial\Omega, t\in(0,\rho)\} &= \{x\in\R^n| \dist(x,\partial\Omega) < \rho\} \setminus \bar{\Omega} = \Omega_{\rho}'. 
\end{align*}
This means that for each $x\in \Omega_{\delta}$, there exists a unique closest point on the boundary $\pO$, and for each
$y\in\pO$ there exists $x\in \overline{\Omega_{\delta}}$ such that $\dist(x,y)=\delta$ and $B(x,\delta)\subset \Omega$.
Furthermore, since $\Omega \setminus \Omega_{\delta}$ is connected,
for each two points $y_1,y_2\in\pO$ and the corresponding
points $x_1,x_2 \in \overline{\Omega_{\delta}}$ such that $\dist(x_1,y_1)=\dist(x_2,y_2)=\delta$ and $B(x_1,\delta), B(x_2,\delta)\subset \Omega$,
there exists a curve $\Gamma_{x_1,x_2}$ in $\Omega\setminus \Omega_{\delta}$ with endpoints $x_1, x_2$. If we look at $\Gamma_{x_1,x_2}$ as a set of points
in $\Omega$, then 
$$\{x\in \R^n: \dist(x,\Gamma_{x_1,x_2})<\delta \}\subset \Omega.$$ 
We will need these properties later.
\end{notation}

We want to extend the function $u$ defined on $\tO \cup \pO$ to 
$$\tO\cup\pO\cup\tO' = D,$$ 
so that
the boundary $\pO$ becomes a hypersurface in $D$. In order to do so, define 
$$
v(x):= u(x) e^{\lambda d(x)}
\ \ \ \text{ for } x\in \tO\cup \pO, 
$$
where $d(x)=\dist(x,\partial\Omega)$ is the distance function.
Fix this notation throughout the rest of this section.

Note that $v(x)=0$ if and only if $u(x)=0$ and $v(x)=u(x)$ on $\pO$.
Since $\partial\Omega$ is $C^2$, so is $d(x)$ on $\tO$. 
For $y\in\pO$ and $x=y+t\nu(y)\in\tO$, we have 
\begin{equation} \label{dist}
\begin{aligned}
\nabla d(x) &= \nabla d(y)=-\nu(y), \\
\Delta d(x) &= -\sum_{i=1}^{n-1} \frac{\kappa_i(y)}{1-\kappa_i(y) d(x)},
\end{aligned}
\end{equation}
where $\{\kappa_i\}$ are the principal curvatures of $\pO$, see the Appendix, pp.~381-383 in \cite{trudinger}.
From \eqref{main} we get that $v$ satisfies the equation 
\begin{equation} \label{main_v}
\begin{aligned}
\Div (A(x) \nabla v) 
+b(x) \cdot \nabla v +c(x) v 
&=0 &\text{ in } \tO, \\
\frac{\partial v}{\partial \nu} &= 0 &\text{ on } \pO,
\end{aligned}
\end{equation}
where 
\begin{equation} \label{abc}
\left.
\begin{array}{ll}
A=I \\
b= -2\lambda \nabla d \\
c= \lambda^2-\lambda \Delta d
\end{array}
\right\}
\text{ in } \tO.
\end{equation}

Write each $x\in\tO$ as $x=y+t\nu(y)$, 
$y\in\pO$ and $t\in(0, -\delta)$, and consider the reflection map
\begin{align*}
\Psi&: \ \tO \to \tO' \\
\Psi(x)&= y-t\nu(y).
\end{align*}
Since $\pO$ is $C^2$, $\Psi$ is $C^2$. We will also use the notation $\Psi(x)=x'$. 
For $x'\in \tO'$ define
$$
v(x'):= v(\Psi^{-1}(x'))=v(x).
$$
Then $v\in C^2(\tO')$, and since $\frac{\p v}{\p \nu}=0$ on $\pO$, we also have that
$\nabla v$ is Lipschitz in $D$, and hence $v$ is twice weakly differentiable in $D$.
By \eqref{main_v}, $v$ satisfies
\begin{equation} \label{main_tO'}
\Div (A \nabla v ) + b\cdot \nabla v +cv=0  \ \ \text{ in } \tO',
\end{equation}
with $A=(a^{ij})_{i,j=1}^n$,
\begin{equation} \label{abc2}
\begin{aligned}
 a^{ij}(x') &=\sum_{k=1}^n \frac{\p \Psi^i}{\p x_k}(x) \frac{\p \Psi^j}{\p x_k}(x) \\
b^i(x') &= - \frac{\p}{\p x_j'} a^{ij}(x') + \Delta \Psi^i(x) +\nabla\Psi^i(x)\cdot b(x) \\
c(x') &= c(x).
\end{aligned} 
\end{equation}
Hence $v$ satisfies
\begin{equation} \label{main_D}
\Div (A \nabla v ) + b\cdot \nabla v +cv=0  
\end{equation}
$a.e.$ in $D$,
where by \eqref{abc} and \eqref{abc2} we have the bounds
\begin{equation} \label{abc_infty}
\begin{aligned}
||A||_{L^{\infty}(D)} &\leq C, \\
||b||_{L^{\infty}(D)} &\leq C \lambda, \\
||c||_{L^{\infty}(D)} &\leq C \lambda^2,
\end{aligned}
\end{equation}
with the constant $C$ depending only on the domain $\Omega$. Since $v$ is also twice weakly differentiable, it
satisfies equation \eqref{main_D} 
in the strong sense (as in \cite[Chapter 9]{trudinger}). 
Next, we show that $A$ is Lipschitz across $\pO$, and hence
$v$ is also a weak solution of \eqref{main_D}. 

Note that 
\begin{equation} \label{A_on_pOmega}
A(x')=\nabla \Psi (x)(\nabla \Psi)^T(x) \ \ \text{ for } x'\in\tO'.
\end{equation}

\begin{prop}
$A$ is uniformly Lipschitz in $D$, with the Lipschitz constant depending
only on $\Omega$. 
\end{prop}

\begin{proof}
This is clear in $\tO$ and $\tO'$; what we are concerned about is the hypersurface $\pO$.  
Let us show that 
\begin{equation} \label{identity}
\nabla \Psi (x) (\nabla \Psi )^T (x) =I \ \ \ \ \ \text{ for }  x\in\pO,
\end{equation}
where 
we take $\Psi (x)=x$ for $x\in \pO$ (and consider the gradient of $\Psi$ in points on $\pO$
only from the side of within $\Omega$).

Take any $x_0\in\pO$. If the tangent plane to $\pO$ at $x_0$ is parallel to the plane $\{x_n=0\}$, then
it is easy to see that 
$$\nabla \Psi (x) = \begin{pmatrix}
                    1 & 0 & \cdots & 0 & 0 \\
                    0 & 1 & \cdots & 0 & 0 \\ 
                     \vdots & & \ddots & & \vdots \\
                    0 & 0 & \cdots & 1 & 0 \\
                    0 & 0 & \cdots & 0 &-1
                    \end{pmatrix}
 $$ 
and indeed $\nabla \Psi (x_0) (\nabla \Psi )^T (x_0) =I$.  If the tangent plane to $\pO$ at $x_0$ is general,
there exists a rotation $R$ which maps this tangent plane to a tangent plane parallel to $\{x_n=0\}$. Then 
by the change of coordinates $y=Rx$, the mapping $\tPsi(y)= R \Psi (R^Ty)$ will be exactly like $\Psi$
in the special case discussed, and we will have
\begin{align*}
I&= \nabla \tPsi (y_0) (\nabla \tPsi )^T (y_0) = R   \nabla \Psi (x_0) R^T R  (\nabla \Psi )^T (x_0) R^T = \\ 
&= R \nabla \Psi (x_0) (\nabla \Psi )^T (x_0) R^T, \\
I &= \nabla \Psi (x_0) (\nabla \Psi )^T (x_0).
\end{align*}

Hence \eqref{identity} is true, and we can define $A(x)=I=\nabla \Psi (x) (\nabla \Psi )^T (x)$ for $x\in\pO$.
Then by \eqref{abc} $A$ is Lipschitz in $\overline{\tO}$, by \eqref{abc2}, $A$ is Lipschitz in $\overline{\tO'}$,
and hence $A$ is Lipschitz in the whole $D$.

\end{proof}

\begin{prop}
Equation \eqref{main_D} is uniformly elliptic in $D$, the ellipticity constant depending only on $\Omega$.
\end{prop}
This is clear for all $x\in\tO$,
and for $x\in\tO'$ we have by \eqref{A_on_pOmega}
$$
\xi^T A(x') \xi = \xi^T \nabla \Psi (x)(\nabla \Psi)^T(x) \xi = |(\nabla \Psi)^T(x) \xi|^2 \geq \frac{1}{\eta^2} |\xi|^2,
$$
where
$$
\eta = \sup_{x\in\tO} |(\nabla \Psi)(x)^{-1}|= \sup_{x'\in\tO'} |\nabla \Psi^{-1}(x')|
$$
depends only on $\Omega$.

\medskip
\noindent
\textbf{Doubling condition for $v$:}

\noindent
To prove Theorem \ref{thm_doubling_bdry}, we will first derive an analogous doubling condition
on the solid $D$ for $v$, and then use the approach as in \cite{lin_nodal} (used for parabolic equations)
to deduce Theorem \ref{thm_doubling_bdry}:

\begin{thm} \label{thm_doubling}
 There exist constants $r_0, C>0$ depending only on $\Omega$ and $n$ such that 
for any $r\leq r_0/2\lambda$ and $x\in \pO$,
\begin{equation} \label{doubling_v}
\int_{B(x,2r)}  v^2 \leq  2^{C\lambda^5} \int_{B(x,r)}  v^2.
\end{equation}
\end{thm}

The main difficulty we encounter is that for a general elliptic equation, an adaptation of the results in \cite{lin_garofalo_2} gives us a bound
for the frequency only on small balls of radius of order $1/\lambda$. That globally translates into
a suboptimal exponential bound. 
Therefore we move on from the more general equation \eqref{main_D}, 
which is satisfied by $v$ in the neighborhood of $\pO$, back to the original Laplace equation inside of $\Omega$, and then back to the neighborhood of the boundary 
with the equation \eqref{main_D} again.


\subsection{Frequency for $v$} \label{frequency_for_v}
\medskip

Near the boundary $\pO$, we will consider the equation \eqref{main_D}. Hence
we use the  generalization of the frequency \eqref{frequency} for more general
elliptic equations from Subsection \ref{sec_frequency_gen}. Let us apply the results
from that subsection to our function $v$.

Recall Theorem \ref{N_L_infty}: for a solution $w$ to a general elliptic equation on $B_1$,
it gives us the bound for the frequency of the form $N(R_1) \leq c_1+ c_2 N(R_2)$,
if $0<R_1 < R_2 \leq r_0$.  
However, it
does not give an explicit bound for $r_0, c_1, c_2$ in terms of
the $L^{\infty}$ bounds on the coefficients of the equation. 
If we use the bounds \eqref{abc_infty} instead of \eqref{L-infinity} and go through
the proof of the theorem, we obtain that \eqref{N_L_infty_bound} holds for $r_0 \sim  {1/\lambda}$,
$c_1, c_2 \sim 1$. We can obtain these bounds easier from Theorem \ref{N_L_infty} by scaling.

\medskip

\noindent
\textbf{Scaling the equation:}

\noindent
Consider equation \eqref{main_D}
in a ball $B(x_0,r_1/\lambda)\subset D$ and define
\begin{equation*}
 v_{x_0,\lambda} (x):= v(x_0+x/\lambda) \ \  \text{ for } x\in B(0,r_1).
\end{equation*}
Then $v_{x_0,\lambda}$ satisfies 
\begin{equation} \label{equation_v_lambda}
\Div (A_{x_0,\lambda} \nabla v_{x_0,\lambda} ) + b_{x_0,\lambda} \cdot \nabla v_{x_0,\lambda} +c_{x_0,\lambda} v_{x_0,\lambda}=0  \ \  a.e.\text{ in } B(0,r_1),
\end{equation}
with
\begin{equation*}
\begin{aligned}
A_{x_0,\lambda}(y) &= A(x_0+y/\lambda), \\
b_{x_0,\lambda}(y) &= \lambda^{-1}b(x_0+y/\lambda ), \\
c_{x_0,\lambda}(y) &= \lambda^{-2}c(x_0+y/\lambda ), \\
\end{aligned}
\end{equation*}
so by \eqref{abc_infty}, $A_{x_0,\lambda}(y)$, $b_{x_0,\lambda}(y)$ and $c_{x_0,\lambda}(y)$
are bounded uniformly in $L^{\infty}$ by a constant depending only on $\Omega$. 
The ellipticity constant of $A$ does not change and the Lipschitz constant of $A$
only improves (since $\lambda \geq 1$), so we also have bounds on them depending only on $\Omega$.
Note that for $x_0\in \overline{\Omega}\cap D$, $A(x_0)=A_{x_0,\lambda}(0)=I$.
A simple change of variables yields that for $x_0\in \overline{\Omega}\cap D$, $y\in B(0,r_1)$ and $r\leq r_1-|y|$,
\begin{equation*}
\begin{aligned}
H_{v_{x_0, \lambda}}(y,r) &= \lambda^{n-1} H_v(x_0+y/\lambda, r/\lambda), \\
D_{v_{x_0, \lambda}}(y,r) &= \lambda^{n-2} D_v(x_0+y/\lambda, r/\lambda), \\
I_{v_{x_0, \lambda}}(y,r) &= \lambda^{n-2} I_v(x_0+y/\lambda, r/\lambda), \\
N_{v_{x_0, \lambda}}(y,r) &= N_v(x_0+y/\lambda, r/\lambda).
\end{aligned}
\end{equation*}

Next, we apply Theorem \ref{N_L_infty},
Proposition \ref{prop_H_prime_gen}, Corollary \ref{cor_reverse_doubling_gen}, Theorem \ref{thm_doubling_gen} and Lemma \ref{lemma_freq_from_doubling_gen}
to the function $v_{x_0,\lambda}$, and using the scaling above we rewrite the results in terms of $v$. We immediately obtain the following results.

\begin{prop} \label{N_L_infty_v}
Let $x_0\in \overline{\Omega}\cap D$, $B(x_0,r_1/\lambda)\subset D$.
Then there exist constants $c_1,c_2>0$, $r_0\in(0,r_1)$ depending only on $r_1$ and $\Omega$ such that 
\begin{equation} \label{N_L_infty_bound_v}
N_v(x_0,R_1) 
\leq c_1 + c_2 N_v(x_0,R_2) \ \  \text{ for any } 0<R_1< R_2 \leq r_0/\lambda. 
\end{equation}
\end{prop}

\begin{prop} \label{thm_doubling_v}
Let $x_0\in \overline{\Omega}\cap D$, $B(x_0,r_1/\lambda)\subset D$.
Then there exist constants $C, c_1,c_2>0$, $r_0\in(0,r_1)$ depending only on $r_1$ and $\Omega$ such that 
for any $0<R_1<R_2<r_0/\lambda$,
\begin{align}
 \dashint_{\p B(x_0,R_2)} v^2 \leq C \left( \frac{R_2}{R_1} \right)^{c_1+c_2N_v(x_0,R_2)}    \dashint_{\p B(x_0, R_1)} v^2,   \label{H_doubling_v} \\
 \dashint_{ B(x_0,R_2)} v^2 \leq C \left( \frac{R_2}{R_1} \right)^{c_1+c_2N_v(x_0,R_2)} \dashint_{ B(x_0, R_1)} v^2.          \label{loc_doubling_v}
\end{align}
\end{prop}

\begin{lemma} \label{lemma_freq_from_doubling_v}
Let $x_0\in \overline{\Omega}\cap D$, $B(x_0,r_1/\lambda)\subset D$.
Assume
\begin{equation} \label{assume_v}
\dashint_{B(x_0,\zeta r)} v^2 \geq \kappa \dashint_{B(x_0,r)} v^2
\end{equation}
for some $\kappa, \zeta \in (0,1)$ and $r\in (0,r_0/\lambda]$, where $r_0$ depends on $r_1$ and $\Omega$ 
and is chosen so that 
Propositions \ref{N_L_infty_v} and \ref{thm_doubling_v}, and the analogies of
Proposition \ref{prop_H_prime_gen} and Corollary \ref{cor_reverse_doubling_gen}
hold.
Then there exists a constant $C_{\zeta}>0$ depending on $r_1$, $\Omega$ 
and $\zeta$ such that
\begin{equation} \label{freq_bound_v}
N_v(x_0,\zeta r) \leq C_{\zeta} (1- \log \kappa).
\end{equation}
In particular, for any $\beta< \zeta $, there exist constants $C_1$, $C_2$  depending on $r_1$, $\Omega$ 
and $\zeta$, $\beta$ such that
\begin{equation}\label{doubling_ab_v}
\dashint_{B(x_0,\beta r)} v^2 
\geq  
\frac{1}{C_1} \kappa^{C_2}
\dashint_{B(x_0,\zeta r)} v^2.
\end{equation}
\end{lemma}

\subsection{Proof of Theorem \ref{thm_doubling}}

We will use an argument of the type as in Proposition \ref{chain_frequency} in a chain.
For that, we will need to start in a point where the frequency is reasonably bounded/where we have
a doubling condition with a reasonable constant. We will require even somewhat more: the integral
of $u^2$ over a small ball with center in this special starting point $y_* \in\pO$
will control the global integral of $u^2$. Recall 
the Notation \ref{delta}: $\delta>0$ is 
a constant depending on $\Omega$ such that each point in $\tO=\{x\in\R^n| \dist(x,\partial\Omega) < \delta\} \cap\Omega$
has a single closest point on $\pO$. 

\begin{lemma} \label{special_point}
For any $\rho< \delta$, there exists a point $y_*\in\pO$ such that for some constant $C$
depending only on $\Omega$, there holds
\begin{equation*}
 \int_{\Omega
} u^2 \leq C \rho^{-2n+1} \int_{B(y_*, \rho)\cap \Omega} u^2 .
\end{equation*}
\end{lemma}
\noindent In the proof we use this interior estimate for harmonic functions:
\begin{prop} \label{interior_harmonic}
Let $w$ be a harmonic function in $B_1$. Then  
\begin{equation*}
 \sup_{B_{1/2}} |w| \leq C \left(\int_{B_1} w^2 \right)^{1/2},
\end{equation*}
where $C$ is a constant depending only on $n$. 
\end{prop}

\begin{proof}
 This can be found e.g.~in \cite{lin_skripta}, Remark 1.19.
\end{proof}

\begin{proof}[Proof of Lemma \ref{special_point}]
Choose a $\rho/2$-net of points $y_1,y_2,\dots,y_m\in\pO$, i.e.~for each
$y\in\pO$ there exists $1\leq i\leq m$ such that $|y-y_i|\leq\rho/2$.
We can always make $m\leq C \rho^{-(n-1)}$, where $C$ depends only on $\Omega$.
Then the balls $\{B(y_i,\rho) \}_{i=1}^m$ cover $\Omega_{\rho/2}$,
and therefore there exists $y_*\in\{y_1,y_2,\dots, y_m\}$ such that
\begin{equation} \label{special_ball}
 \int_{\Omega_{\rho/2}} u^2 \leq C \rho^{-(n-1)} \int_{B(y_*, \rho)\cap \Omega} u^2.
\end{equation}
Now we just need to bound $\int_{\Omega
} u^2$ in terms of $\int_{\Omega_{\rho/2}} u^2$.
Let 
\begin{equation*}
\rhoO=\left\{x\in\Omega: \dist(x,\pO)\geq \frac{\rho}{4} \right\} = \Omega\setminus \Omega_{\frac{\rho}{4}}.
\end{equation*}
Then $\prhoO= \left\{x\in\Omega: \dist(x,\pO)= \rho/4 \right\}$ and
for any $x\in \prhoO$, by the scaling of the interior estimate in Proposition \ref{interior_harmonic}
we obtain
\begin{equation*}
 \sup_{B(x,\rho/8)} u^2 \leq C \rho^{-n} \int_{B(x,\rho/4)} u^2 \leq C \rho^{-n} \int_{\Omega_{\rho/2}} u^2.
\end{equation*}
Hence, by the maximum principle for the harmonic function $u$ inside $\rhoO$,
\begin{align*}
 \sup_{\rhoO} u ^2 &\leq \sup_{\prhoO} u ^2 \leq \sup_{x\in \prhoO} \sup_{B(x,\rho/8)} u ^2 \leq C \rho^{-n} \int_{\Omega_{\rho/2}} u^2, \\
\int_{\Omega
} u^2 &= \int_{\Omega
\setminus \Omega_{\rho/2}} u^2 + \int_{\Omega_{\rho/2}} u^2 \leq
C 
\sup_{\rhoO} u ^2 + \int_{\Omega_{\rho/2}} u^2  \leq C' \rho^{-n} \int_{\Omega_{\rho/2}} u^2.
\end{align*}
Combining this with \eqref{special_ball}, we obtain the desired estimate.
\end{proof}


\begin{proof}[Proof of Theorem \ref{thm_doubling}]
Fix $r_1:= \delta \lambda_1$, where $\lambda_1$ is the smallest positive Steklov eigenvalue for $\Omega$. Then
$B(x_0, r_1/\lambda) \subset D$ for every $x_0\in \pO$ and every Steklov eigenvalue $\lambda$. Fix $r_0<r_1$
so that Propositions \ref{N_L_infty_v} and \ref{thm_doubling_v} and Lemma \ref{lemma_freq_from_doubling_v}
hold with this $r_0$. Note that $r_1$ and $r_0$ depend only on $\Omega$.

To prove the theorem, it is enough to prove that there exists a constant 
$C$ depending only on $\Omega$ and $n$ such that for every $y\in\pO$
\begin{equation} \label{doubling_v_r_0}
\int_{B(y,r_0/\lambda)}  v^2 \leq  2^{C\lambda^5} \int_{B(y,r_0/2\lambda)}  v^2.
\end{equation}
Indeed, this will prove \eqref{doubling_v} for $r=r_0/2\lambda$. Once we know it for this $r$, from
Lemma \ref{lemma_freq_from_doubling_v} we obtain a bound for the frequency $N(r_0/2\lambda)\leq C \lambda^5$,
and from Propositions \ref{N_L_infty_v} and \ref{thm_doubling_v} we obtain \eqref{doubling_v} for any $r\leq r_0/2\lambda$.

The rest of the proof is organized as follows. First, we prove a doubling condition for $v$ in a special point $y=y_*\in\pO$,
move to a close-by point inside $\Omega$, and switch to $u$. As the next step, we show a doubling condition for $u$ in a point
near $y_*$ for a ball with fixed radius (independent of $\lambda$), and propagate the doubling condition estimate
through $\Omega$ into a neighborhood of any other given point $y_0\in\pO$. In step 3, we gradually pass from the ball with fixed radius
to a small ball in  $\sim 1/\lambda$-neighborhood of $\pO$, still for $u$ and within $\Omega$.
In the last step, we switch back to $v$ and deduce the doubling condition \eqref{doubling_v_r_0}
in the point $y_0\in\pO$.

\noindent
{\bf Step 1:}

We start in a point $y=y_*\in\pO$, which we obtain from Lemma \ref{special_point}
for $\rho= r_0/2\lambda$, i.e.~there holds
\begin{equation*}
 \int_{\Omega
} u^2 \leq C \left(\frac{r_0}{2\lambda}\right)^{-2n+1} \int_{B(y_*, r_0/2\lambda)\cap \Omega} u^2.
\end{equation*}
In $\Omega$, we have $|v(x)|=|u(x)e^{\lambda d(x)}| \geq |u(x)|$. Hence
\begin{equation} \label{special_point_v}
  \int_{B(y_*, r_0/2\lambda)} v^2 \geq
 \int_{B(y_*, r_0/2\lambda)\cap \Omega} u^2
\geq \frac{1}{C} \left( \frac{r_0}{2\lambda} \right)^{2n-1}  \int_{\Omega} u^2.
\end{equation}
On the other hand, we show that the integral $\int_{B(y,r_0/\lambda)} v^2$
is controlled by $\int_{\Omega} u^2$ for any $y\in \pO$. Fix $y\in \pO$.
Notice that $r_0/\lambda < \delta$. Recall that the reflection map $\Psi$ introduced in Subsection \ref{reflection}
maps any $x\in \Omega_{\delta}$, written as $x=z+t\nu(z)$, where $z\in\pO$, into $\Psi(x)=z-t\nu(z)$,
and this is a $1$-to-$1$ map from $\Omega_{\delta}$ onto $\Omega_{\delta}'=D\setminus \overline{\Omega}$.
Since $B(y, r_0/\lambda)\setminus \overline{\Omega} \subset \Omega_{r_0/\lambda}' \subset \Omega_{\delta}'$, the 
map $\Psi^{-1}$ maps $B(y, r_0/\lambda)\setminus \overline{\Omega} $ onto a subset of $\Omega_{r_0/\lambda}$,
and by change of variables, using the boundedness of the Jacobian of $\Psi^{-1}$ (the bound depends
only on $\Omega$), we obtain
\begin{equation} \label{hviezdicka}
 \int_{B(y, r_0/\lambda)\setminus \overline{\Omega}} v^2 \leq C \int_{\Omega_{r_0/\lambda}} v^2.
\end{equation}
On $\Omega_{r_0/\lambda}$ we have 
\begin{equation} \label{v_u_relationship}
|v(x)|=|u(x) e^{\lambda d(x)}| \leq |u(x)|e^{r_0} \leq C |u(x)|,
\end{equation}
with $C$ depending only on $\Omega$. Hence, using \eqref{hviezdicka} and \eqref{v_u_relationship},
\begin{equation} \label{any_point_v}
 \int_{B(y,r_0/\lambda)} v^2 
\leq C \int_{\Omega_{r_0/\lambda}} v^2 
 \leq C' \int_{\Omega_{r_0/\lambda}} u^2 \leq C' \int_{\Omega} u^2.
\end{equation}
Now we just need to propagate the estimate \eqref{special_point_v} into any point $y\in\pO$ -- combined with
\eqref{any_point_v}, we will obtain a doubling condition for $v$.

Using \eqref{special_point_v} and \eqref{any_point_v} in $y=y_*$, we know that
\begin{equation*}
  \int_{B(y_*,r_0/2\lambda)} v^2 \geq \frac{1}{C} \frac{1}{\lambda^{2n-1}} \int_{B(y_*,r_0/\lambda)} v^2.
\end{equation*}
Hence, from Lemma \ref{lemma_freq_from_doubling_v} used for $r=\frac{r_0}{\lambda}$, $\zeta=\frac{1}{2}$, $\kappa= \frac{1}{C} \lambda^{-2n+1}$, $\beta = \frac{1}{4}$, we get
\begin{equation} \label{doubling_r_4}
 \int_{B(y_*,r_0/4 \lambda)} v^2 \geq \frac{1}{C_1} \frac{1}{\lambda^{C_2}} \int_{B(y_*,r_0/2\lambda)} v^2,
\end{equation}
where $C_1, C_2$ depend only on $\Omega$. 

Now fix $y=y_0\in \pO$ in which we want to prove \eqref{doubling_v_r_0}.
Recalling Notation \ref{delta}, we know that there exist points $z_1, z_2\in \Omega$ such that
$\dist(z_1, y_*)=\dist(z_2,y_0)=\delta$, $B(z_1,\delta),B(z_2, \delta) \subset \Omega$, and there is
a curve $\Gamma$ in $\Omega$ with endpoints $z_1, z_2$ such that if we 
look at $\Gamma$ as a set of points
in $\Omega$, then $\{x\in \R^n: \dist(x,\Gamma)<\delta \}\subset \Omega$. 
We will propagate the doubling condition estimate along this curve. First choose $x_1$
on the segment $y_*z_1$ such that $\dist(y_*,x_1)= r_0/4\lambda$. Then
$B(x_1,r_0/4\lambda)$ lies inside $\Omega$ and touches $\pO$ in $y_*$.
We also have 
$$B(y_*,r_0/4\lambda) \subset B(x_1,r_0/2\lambda) \subset B(x_1,3r_0/4\lambda) \subset B(y_*,r_0/\lambda).$$
Hence, using \eqref{doubling_r_4} and \eqref{special_point_v},
\begin{equation} \label{x_1_r_0_2}
\int_{B(x_1,r_0/2\lambda)} v^2 
\geq  \int_{B(y_*,r_0/4\lambda)} v^2 
\geq\frac{1}{C_1'} \frac{1}{\lambda^{C_2'}} \int_{B(y_*,r_0/2\lambda)} v^2 
\geq \frac{1}{C_1}  \frac{1}{\lambda^{C_2}}  \int_{\Omega} u^2.
\end{equation}
On the other hand, from $B(x_1,3r_0/4\lambda) \subset B(y_*,r_0/\lambda)$ and \eqref{any_point_v} (used for $y=y_*$), we obtain
\begin{equation*}
 \int_{ B(x_1,3r_0/4\lambda)} v^2 \leq \int_{B(y_*,r_0/\lambda)} v^2 \leq C \int_{\Omega} u^2,
\end{equation*}
so together with \eqref{x_1_r_0_2} we have
\begin{equation*}
 \int_{B(x_1,r_0/2\lambda)} v^2 \geq \frac{1}{C_1}  \frac{1}{\lambda^{C_2}} \int_{ B(x_1,3r_0/4\lambda)} v^2.
\end{equation*}
Hence it follows from Lemma \ref{lemma_freq_from_doubling_v} used in the center $x_1$ for $r=3r_0/4\lambda$,
$\zeta=2/3$, $\kappa =\frac{1}{C_1}  \frac{1}{\lambda^{C_2}}$, $\beta = 1/6$, that
\begin{equation*} 
 \int_{B(x_1,r_0/8\lambda)} v^2 \geq \frac{1}{\widetilde{C_1}}  \frac{1}{\lambda^{\widetilde{C_2}}} \int_{B(x_1,r_0/2\lambda)} v^2 \geq \frac{1}{C_1}  \frac{1}{\lambda^{C_2}}  \int_{\Omega} u^2,
\end{equation*}
where in the last inequality we used \eqref{x_1_r_0_2}. Notice that starting from the estimate \eqref{special_point_v} for a ball with
center $y_*\in\pO$, we moved to an estimate for a ball $B(x_1,r_0/8\lambda)$ inside $\Omega$. Since $B(x_1,r_0/8\lambda)\subset \Omega_{r_0/\lambda}$,
using \eqref{v_u_relationship} we can move from $v$ towards the harmonic function $u$:
\begin{equation} \label{doubling_x_1}
  \int_{B(x_1,r_0/8\lambda)} u^2 \geq \frac{1}{C} \int_{B(x_1,r_0/8\lambda)} v^2 \geq \frac{1}{C_1}  \frac{1}{\lambda^{C_2}}  \int_{\Omega} u^2.
\end{equation}

\noindent
{\bf Step 2:}

Now we are ready to move towards balls with fixed radii which will not depend on $\lambda$. Take the point $z_1$ for which
$B(z_1,\delta)$ lies inside $\Omega$ and touches $\pO$ in $y_*$. Since $x_1$ lies on the segment $y_*z_1$ in distance $r_0/4\lambda$ from $y_*$, we have
$$
B(x_1,r_0/8\lambda) \subset B(z_1,\delta - r_0/8\lambda) \subset B(z_1,\delta) \subset \Omega.
$$
Hence, by \eqref{doubling_x_1} we obtain
\begin{equation} \label{doubling_z_1_a}
 \int_{B(z_1,\delta - r_0/8\lambda)} u^2 \geq \int_{B(x_1,r_0/8\lambda)} u^2 \geq \frac{1}{C_1}  \frac{1}{\lambda^{C_2}}  \int_{\Omega} u^2
\geq \frac{1}{C_1}  \frac{1}{\lambda^{C_2}}  \int_{B(z_1,\delta)} u^2.
\end{equation}
Now we can use Lemma \ref{lemma_freq_from_doubling} for the harmonic function $u$ on the ball $B(z_1,\delta)$ and $\alpha=1-\frac{r_0}{8\lambda \delta}$,
$\theta=1 -\frac{r_0}{16\lambda \delta} $, $\kappa = \frac{1}{C_1}  \frac{1}{\lambda^{C_2}} $, $\beta = 1/2$. For this choice, we have
\begin{align*}
 1 -\theta^n \geq \frac{1}{C\lambda}, \\
\log ({\theta}/{\alpha}) \geq \frac{1}{C\lambda}, \\
\log (\alpha/\beta) \leq \log 2,
\end{align*}
and Lemma \ref{lemma_freq_from_doubling} together with \eqref{doubling_z_1_a} yield
\begin{equation}\label{doubling_z_1}
\int_{B(z_1,\delta/2)} u^2 \geq \left (\frac{1}{C_1}  \frac{1}{\lambda^{C_2}}\right)^{\widetilde{C}\lambda} \int_{B(z_1,\delta - r_0/8\lambda)} u^2 \geq
\left(\frac{1}{2} \right)^{C\lambda \log \lambda} \int_{\Omega} u^2.
\end{equation}

Now we propagate the estimate \eqref{doubling_z_1} along $\Gamma$ into the point $z_2$.
We choose points 
$z^1=z_1, z^2, \dots, z^k=z_2$ on $\Gamma$ such that 
\begin{equation*}
\dist (z^i,z^{i+1})\leq \delta/4 \ \ \ \text{ for } i=1,2,\dots,k-1. 
\end{equation*}
Notice that $k$ is bounded by a constant depending only on $\Omega$. 
Assume
\begin{equation} \label{doubling_induction}
 \int_{B(z^i,\delta/2)} u^2 \geq \left(\frac{1}{2} \right)^{C_i\lambda \log \lambda} \int_{\Omega} u^2
\end{equation}
for some constant $C_i$ depending only on $\Omega$ -- for $i=1$, this is \eqref{doubling_z_1}.
Then
the inclusion $B(z^{i+1}, \delta/2) \supset B(z^i, \delta/4)$, Lemma \ref{lemma_freq_from_doubling} 
and \eqref{doubling_induction} imply
\begin{equation*}
\begin{aligned}
 \int_{B(z^{i+1},\delta/2)} u^2 &\geq \int_{B(z^i,\delta/4)} u^2  \\
&\geq \left(\frac{1}{2} \right)^{C\lambda \log \lambda}  \int_{B(z^i,\delta/2)} u^2 
\geq \left(\frac{1}{2} \right)^{C_{i+1}\lambda \log \lambda} \int_{\Omega} u^2
\end{aligned}
\end{equation*}
(here $C$ comes from Lemma \ref{lemma_freq_from_doubling} and depends on $C_i$), which is \eqref{doubling_induction} for $i+1$ instead of $i$.
After $k-1$ steps, the constant $C_k$ will be bounded by a constant depending only on $\Omega$, and we obtain
\begin{equation}\label{doubling_z_2}
\int_{B(z_2,\delta/2)} u^2 \geq 
\left(\frac{1}{2} \right)^{C\lambda \log \lambda} \int_{\Omega} u^2.
\end{equation}

\noindent
{\bf Step 3:}

Now we need to get from the ball $B(z_2,\delta/2)$ closer to the boundary, where we will be able to switch
back to the function $v$. In each step, we multiply the distance from $y_0$ by a factor of $5/6$.
Hence, we will need $\sim \log \lambda$ steps to get to a neighborhood of $\pO$ of order $1/\lambda$.

Recall that $B(z_2,\delta)$ lies inside $\Omega$ and touches $\pO$ in $y_0$. Consider the sequence
$x^0, x^1, \dots, x^l$ of points on the segment $z_2y_0$ such that 
$\dist(x^i,y_0) = \delta\left(\frac{5}{6}\right)^{i}$
(i.e.~$x^0=z_2$) and $l$ is chosen to be the smallest integer such that $\dist(x^l,y_0)\leq \frac{r_0}{4\lambda}$, i.e.
\begin{equation*} 
l= \frac{\log \lambda}{\log 6/5} + C,
\end{equation*}
where $\frac{\log 4\delta - \log r_0}{\log 6/5} \leq C < \frac{\log 4\delta - \log r_0}{\log 6/5} +1$.
Denote $r^i:=\delta \left(\frac{5}{6}\right)^i $, so that the ball $B(x^i, r^i)$ lies in $\Omega$
and touches $\pO$ in $y_0$.

By induction, we will show that
\begin{equation} \label{doubling_x^i}
 \int_{B(x^i,r^i/6 ) } u^2 \geq 
\left(\frac{1}{2} \right)^{ C 2^i \lambda \log \lambda } \tau^{2^i-1}
\int_{\Omega} u^2,
\end{equation}
where $\tau=\left(\frac{25}{24}\right)^n -1$. For $i=0$, this follows from \eqref{doubling_z_2}
and Lemma \ref{lemma_freq_from_doubling} used for $r=\delta$, $\alpha=1/2$, $\beta=1/6$.

Now assume that \eqref{doubling_x^i} holds for some $0\leq i \leq l-1$.
Since $\dist(x^i,x^{i+1})= r^i/6$, we have $B(x^{i+1},2r^{i+1}/5) = B(x^{i+1}, r^i/3)\supset B(x^i, r^i/6)$.
On the other hand, we have $B(x^{i+1}, r^{i+1})\subset \Omega$. Hence, \eqref{doubling_x^i} implies
\begin{equation} \label{doubling_x^i+1}
\begin{aligned}
  \int_{B(x^{i+1},2r^{i+1}/5 ) } u^2 &\geq  \int_{B(x^i,r^i/6 ) } u^2 \geq 
   \left(\frac{1}{2} \right)^{ C 2^i \lambda \log \lambda } \tau^{2^i-1} \int_{\Omega} u^2 \\
  &\geq \left(\frac{1}{2} \right)^{ C 2^i \lambda \log \lambda } \tau^{2^i-1} \int_{B(x^{i+1},r^{i+1} ) } u^2.
\end{aligned}
\end{equation}
Apply Lemma \ref{lemma_freq_from_doubling} in point $x^{i+1}$ for $r=r^{i+1}$ and 
\begin{equation*}
\alpha= 2/5,\ \ \ \   \beta =1/6,\ \ \ \  \theta = 24/25, \ \ \ \ \kappa = \left({5}/{2}\right)^n  \left({1}/{2} \right)^{ C 2^i \lambda \log \lambda } \tau^{2^i-1}.
\end{equation*}
This is chosen so that the exponent in \eqref{doubling_ab} is one:
$\frac{\log (\alpha/\beta)}{\log(\theta/\alpha)} =1. $
We obtain
\begin{equation*}
\begin{aligned}
 &\int_{B(x^{i+1},r^{i+1}/6 ) } u^2  \\
&\geq \frac{(1/6)^n}{(2/5)^n}
 \left( \left(\frac{5}{2}\right)^n \left(\frac{1}{2} \right)^{ C 2^i \lambda \log \lambda } \tau^{2^i-1}\right) \left(1-\left( \frac{24}{25}\right)^n\right)
\int_{B(x^{i+1},2r^{i+1}/5 ) } u^2  \\
&\geq \left(\frac{1}{2} \right)^{ C 2^{i+1} \lambda \log \lambda  }\tau^{2^{i+1}-1} \int_{B(x^{i+1},r^{i+1} ) } u^2
\end{aligned}
\end{equation*}
(the last inequality follows from \eqref{doubling_x^i+1}), which is \eqref{doubling_x^i} for $i+1$.

Using \eqref{doubling_x^i} for $i=l$, we obtain
\begin{equation} \label{doubling_x^l}
 \int_{B(x^l,r^l/6 ) } u^2 
\geq \left(\frac{1}{2} \right)^{ C  \lambda^{1+\frac{\log 2}{\log (6/5)}} \log \lambda } \int_{\Omega} u^2 
\geq \left(\frac{1}{2} \right)^{ C  \lambda^5 } \int_{\Omega} u^2.
\end{equation}

{\bf Step 4:}
Recall that $l$ was chosen so that $r^l=\dist(x^l, y_0)\leq \frac{r_0}{4\lambda}$. Hence,
$B(y_0,r_0/2\lambda)\cap \Omega \supset B(x^l,r^l/6)$, and using that $|v(x)|\geq |u(x)|$ in $\Omega_{\delta}$
and \eqref{doubling_x^l}, we obtain
\begin{equation*}
  \int_{B(y_0,r_0/2\lambda ) } v^2 \geq \int_{B(y_0,r_0/2\lambda)\cap \Omega } u^2 \geq
 \int_{B(x^l,r^l/6 ) } u^2 \geq 
 \left(\frac{1}{2} \right)^{ C  \lambda^5 } \int_{\Omega} u^2.
\end{equation*}
Together with \eqref{any_point_v} used for $y=y_0$ which bounds $\int_{B(y_0,r_0/\lambda)} v^2$
by a constant times $\int_{\Omega} u^2$, we obtain the desired doubling condition \eqref{doubling_v_r_0}
in $y=y_0$, what we wanted.
\end{proof}

\begin{remark}
In Step 3, instead of our explicit choice for decreasing the distance to the boundary by the multiplicative
factor of $5/6$, the radius of the smaller ball in control being $1/6$ of the distance to the boundary,
and the auxiliary constant $\theta = 24/25$, we can do the calculation with general constants. 
While the computation becomes more complicated, we do not gain too much: instead of our
exponent $1+ \frac{\log 2}{\log (6/5)} \approx 4.8$, a numerical calculation suggests that
the optimal choice of the constants gives an exponent of roughly $4.6$.
\end{remark}

\subsection{Doubling condition on $\pO$} \label{sec-doubling-on-bdry}
 
In this section, we will use Theorem \ref{thm_doubling} to prove Theorem \ref{thm_doubling_bdry}. 
%
%
As a connection between the boundary integrals and integrals over solid balls,
we will use a quantitative Cauchy uniqueness theorem
which was proved in \cite{lin_nodal}. 
In \cite{lin_nodal} it is formulated for solutions of general elliptic
equations under the same assumptions as used in Subsection \ref{sec_frequency_gen}:

\begin{lemma}[\cite{lin_nodal}, Lemma 4.3] \label{quant_cauchy}
 Let $w$ be a solution of \eqref{equation} in $B_1^+\subset \R^n$ with 
the coefficients of the equation satisfying \eqref{ellipticity}-\eqref{L-infinity} and
$||w||_{L^2(B_1^+)}\leq 1$.
Suppose that 
\begin{equation*}
 ||w||_{H^1(\Gamma)} + ||(\p w/\p x_n)||_{L^2(\Gamma)}\leq \epsilon \ll 1,
\end{equation*}
where $\Gamma=\{(x',0)\in\R^n : |x'|<3/4\}$. Then
$||w||_{L^2(B_{1/2}^+)}\leq C \epsilon^{\beta}$ for some positive constants $C,\beta$ which
depend only on $n, \alpha, K, \gamma$.
\end{lemma}
Notice that Lemma \ref{quant_cauchy} relates the $H^1$-norm on a hypersurface to the $L^2$-norm
in the solid ball. 
The following lemma gives us the missing connection between the $H^1$-norm 
and the $L^2$-norm on the hypersurface; the additional $H^2$-norm on the solid will
be bounded using interior elliptic estimates. 

\begin{lemma} \label{lemma_trace_r_n-1}
 Let $w\in H^2(\R^n)$ and consider the trace of $w$ onto 
$$\{ x\in\R^n : x_n=0\}=\R^{n-1},$$ 
which we denote by $w$. Then there exists a constant $C$ depending only on $n$ such that for any $\eta>0$,
\begin{equation} \label{trace_estimate}
 ||\nabla w||_{L^2(\R^{n-1})} \leq \eta ||w||_{H^2(\R^n)} + \frac{C}{\eta^2}||w||_{L^2(\R^{n-1})}.
\end{equation}
\end{lemma}

\begin{proof}
First we note that for any $w\in H^{3/2}(\R^{n-1})$, $\eta>0$, there holds
\begin{equation} \label{H^1-H^3/2-L^2}
 ||\nabla w||_{L^2(\R^{n-1})} \leq \eta ||w||_{\dot{H}^{3/2}(\R^{n-1})} +\frac{1}{\eta^2}||w||_{L^2(\R^{n-1})}.
\end{equation}
This estimate is a J.~L.~Lions-type lemma and can be easily proven by using the Fourier transform,
Plancheral identity and Young's inequality.
Next, using the trace theorem $H^{3/2}(\R^{n-1}) \subset H^2(\R^n)$ (continuously),
for any $w\in H^2(\R^n)$ we obtain from \eqref{H^1-H^3/2-L^2}
\begin{equation*}
 ||\nabla w||_{L^2(\R^{n-1})} \leq C \eta  ||w||_{H^2(\R^n)} + \frac{1}{\eta^2}||w||_{L^2(\R^{n-1})}.
\end{equation*}
Using $\eta/C$ instead of $\eta$, we obtain \eqref{trace_estimate}.
\end{proof}

\begin{proof}[Proof of Theorem \ref{thm_doubling_bdry}]
Since $u=v$ on $\pO$, we can rewrite \eqref{doubling_bdry} (which we want to prove for $r$ small enough) as
\begin{equation} \label{doubling_bdry_v}
 \int_{B(x,2r) \cap \pO}  v^2 \leq  2^{C\lambda^5} \int_{ B(x,r) \cap \pO}  v^2.
\end{equation}
Fix the point $x=x_0\in\pO$ in which we want to prove \eqref{doubling_bdry_v}.
We use the scaling from Section \ref{frequency_for_v}: $v_{x_0,\lambda}(x)=v(x_0+x/\lambda)$, $x\in B(0,\delta)$ (here $\delta$ is such that $B(x_0,\delta/\lambda_1)\subset D$).
Then $v_{x_0,\lambda}$ satisfies \eqref{equation_v_lambda} with the coefficients 
 $A_{x_0,\lambda}(y)$, $b_{x_0,\lambda}(y)$ and $c_{x_0,\lambda}(y)$
bounded uniformly in $L^{\infty}$ by a constant depending only on $\Omega$, i.e.~they satisfy
\eqref{ellipticity}-\eqref{L-infinity} with $\alpha,K,\gamma$ independent of $\lambda$ (the Lipschitz constant $\gamma$ only improves
with the scaling, ellipticity $\alpha$ stays the same). Proving 
Theorem \ref{thm_doubling_bdry}
is equivalent to
proving, for any $r<r_1$ (where $r_1$ is to be determined, depending on $\Omega$),
\begin{equation} \label{doubling_bdry_v_lambda}
  \int_{B(0,2r) \cap \pO_{x_0,\lambda}}  v_{x_0,\lambda}^2 \leq  2^{C\lambda^5} \int_{ B(0,r) \cap \pO_{x_0,\lambda}}  v_{x_0,\lambda}^2,
\end{equation}
where $\Omega_{x_0,\lambda} = \{x: x_0+ x/\lambda \in \Omega \}$.
From Theorem \ref{thm_doubling}, we know that there
exist constants $r_0, C$ depending only on $\Omega$ and $n$ such that 
for any $r\leq r_0/2$, 
\begin{equation} \label{doubling_v_lambda} 
\int_{B(0,2r)}  v_{x_0,\lambda}^2 \leq  2^{C\lambda^5} \int_{B(0,r)}  v_{x_0,\lambda}^2.
\end{equation}
Fix this $r_0$.

Since we want to use the lemmas above, we will flatten out the hypersurface $\pO$. Since $\Omega$ is a $C^2$-domain,
we can assume that in $B(x_0,r_0/\lambda)$ it is a graph of a $C^2$-function whose $C^2$-norm is bounded by a constant $M$
independent of $x_0$ and $\lambda$ -- otherwise we just diminish $r_0$. After scaling by $\lambda$, this bound only improves. Hence,
we can assume that
\begin{equation*}
 B(0,r_0) \cap \partial\Omega_{x_0,\lambda} = \{ (x',x_n)\in B(0,r_0): x'\in B^{n-1}(0,r_0), x_n=\Phi(x') \},
\end{equation*}
where 
$\Phi\in C^2(B^{n-1}(0,r_0))$, $\Phi(0)=0$, $\nabla \Phi (0) =0$, 
$||\Phi||_{C^2}\leq M$,
and hence (by possibly diminishing $r_0$) 
\begin{equation} \label{Phi_C_1_bound}
||\Phi||_{C^1(B^{n-1}(0,r_0))}\leq \epsilon,
\end{equation}
$\epsilon \leq 0.1$ 
to be chosen later.
Now 
define the injective map $F:B^{n-1}(0,r_0) \times \R \to \R^n$ 
by
$$
F(x',x_n):= (x',x_n + \Phi(x')).
$$
The inverse $F^{-1}(x',x_n) = (x',x_n - \Phi(x'))$ is well-defined on $B(0,r_0)$. 
Using \eqref{Phi_C_1_bound} one can easily compute that 
for any $r<r_0$ resp. $r<r_0/(1+\epsilon)$,
\begin{gather} 
 B(0,r/(1+\epsilon)) \subset F(B(0,r)) \subset B(0,(1+ \epsilon) r)  \label{F_subset}, \\ 
\begin{aligned}
F(B^{n-1}(0,r/(1+\epsilon))\times\{0\} )  \subset  B(0,r) &\cap \partial \Omega_{x_0,\lambda} \\
&\subset F(B^{n-1}(0,(1+\epsilon) r)\times\{0\} ). \label{F_supset_bdry}
\end{aligned}
\end{gather}
Therefore, to prove \eqref{doubling_bdry_v_lambda}, it is enough to show
\begin{equation} \label{doubling_bdry_v_lambda_F}
  \int_{F(B^{n-1}(0,2r)\times \{0\} )}  v_{x_0,\lambda}^2 \leq  2^{C\lambda^5} \int_{ F(B^{n-1}(0,r)\times \{0\} )}  v_{x_0,\lambda}^2
\end{equation}
for all $r\leq r_1$ for some $r_1\leq r_0/2(1+\epsilon)$ depending only on $\Omega$
(we can use this doubling condition twice to account for the loss of the factor $(1+\epsilon)^2$ in radii).

Now define 
$$w(x):=v_{x_0,\lambda}(F(x)) \ \ \text{ for } x\in B(0,r_0/(1+\epsilon)).$$ 
We apply area formula to both sides of \eqref{doubling_bdry_v_lambda_F} under the map $F|_{B^{n-1}(0,2r)\times\{0\}}$.
The (generalized) Jacobian of this map will be
$J= \sqrt{1+ |\nabla \Phi|^2}$,
so using \eqref{Phi_C_1_bound}, we get $1  \leq J \leq 1+\epsilon \leq 1.1$.
Hence, to prove \eqref{doubling_bdry_v_lambda_F}, it is enough to prove
\begin{equation} \label{doubling_bdry_w}
  \int_{B^{n-1}(0,2r)\times \{0\} }  w^2 \leq  2^{C\lambda^5} \int_{ B^{n-1}(0,r)\times \{0\} }  w^2
\end{equation}
for all $r\leq r_1$ for some $r_1\leq r_0/2(1+\epsilon)$ depending only on $\Omega$. Using the same arguments, 
it follows from \eqref{doubling_v_lambda} that
\begin{equation} \label{doubling_w}
\int_{B(0,2r)}  w^2 \leq  2^{C\lambda^5} \int_{B(0,r)}  w^2 \ \ \text{ for any } r\leq r_0/2(1+\epsilon).
\end{equation}
Using that $v_{x_0,\lambda}$ solves the equation \eqref{equation_v_lambda}, we obtain that $w=v_{x_0,\lambda} \circ F $ solves
\begin{equation} \label{equation_w}
 \Div (\tilde{A} \nabla w ) + \tilde{b}\cdot \nabla w + \tilde{c} w=0,
\end{equation}
where
\begin{equation*}
\begin{aligned}
\tilde{A} &= A_{x_0,\lambda}\circ F +P\circ F,\\
 \tilde{b} &= b_{x_0,\lambda} \circ F +Q \circ F,  \\
\tilde{c} &= c_{x_0,\lambda} \circ F.
\end{aligned}
\end{equation*}
The matrix enries of the perturbation $P$ have the form of a sum of products between $a^{ij}$'s and $\Phi_{x_i}$'s,
and the vector entries of $Q$ are sums of products of $a^{ij}_{x_n}$'s resp.~$b^i$'s and $\Phi_{x_i}$'s.
Hence, 
\begin{equation*}
\begin{aligned}
||P||_{L^{\infty}} &
\leq C K \epsilon, \\
||Q||_{L^{\infty}} &
\leq C (K  + \gamma ) \epsilon,
\end{aligned}
\end{equation*}
and the coefficients of equation \eqref{equation_w} will satisfy $L^{\infty}$-bounds 
depending only on $\Omega$,
the Lipschitz constant of $\tilde A$ as well, and if we choose $\epsilon$ small enough, we also get the ellipticity
of $\tilde{A}$ depending only on $\Omega$ (or even an absolute bound like $\tilde{\alpha} \geq 1/2$).

Now fix $r\leq r_0/2(1+\epsilon)$, so that \eqref{doubling_w} holds.
We use scaling again so that we move to balls of fixed radii: 
for $x\in B(0,2)$, denote
$\tilde w (x):=\tilde{c} w(rx)$, 
where the constant $\tilde{c}$ is chosen so that
\begin{equation} \label{int_2r}
 \int_{B(0,2)} \tilde w^2 =1.
\end{equation}
Using the doubling condition \eqref{doubling_w} twice, we obtain
\begin{equation} \label{int_r/2}
 \int_{B(0,1/2)} \tilde w^2 \geq 2^{-C\lambda^5}.
\end{equation}
Note that $\tilde w$ satisfies an equation of type \eqref{equation}
with the coefficients satisfying \eqref{ellipticity}-\eqref{L-infinity} with $\alpha$, $K$,$\gamma$ depending
only on $\Omega$. 
We will show that
 \begin{equation} \label{H^1-norm}
 ||\tilde w||_{H^1(B^{n-1}(0,3/4 ))} + ||(\p \tilde w/\p x_n)||_{L^2(B^{n-1}(0,3/4 ))} \geq 2^{-C\lambda^5} 
\end{equation}
with 
$C$ still depending only on $\Omega$.
Take the constants $C=C_1,\beta$ from Lemma \ref{quant_cauchy} used for $\tilde w$ which only depend on 
$\Omega$. Then if we have
$$ ||\tilde w||_{H^1(B^{n-1}(0,3/4 ))} + ||(\p \tilde w/\p x_n)||_{L^2(B^{n-1}(0,3/4 ))} = \tilde{\epsilon} \ll 1, $$
then Lemma \ref{quant_cauchy} and \eqref{int_2r} imply
$
|| \tilde w ||_{L^2(B(0, 1/2))} \leq C_1 \tilde\epsilon^{\beta}.
$
Since we have the lower bound \eqref{int_r/2}, it follows that
$ \tilde\epsilon \geq C_1^{-1/\beta} 2^{-C\lambda^5 /2 \beta }, $
so we get \eqref{H^1-norm} with the constant $(C + 2\log_2 C_1 )/2\beta$. 

Now we will show that 
\begin{equation} \label{H^1-normal}
||\nabla \tilde w||_{L^2(B^{n-1}(0,3/4 ))} \geq \frac{1}{C}
|| (\p \tilde w/\p x_n)||_{L^2(B^{n-1}(0,3/4 ))},  
\end{equation}
and hence
from \eqref{H^1-norm} 
we will get a lower bound for $||\tilde w||_{H^1(B^{n-1}(0,3/4 ))}$.
After that, we will show that $||\tilde w||_{L^2(B^{n-1}(0,7/8 ))}$ controls $||\tilde w||_{H^1(B^{n-1}(0,3/4 ))}$,
and hence we will get a lower bound for it, which will be a significant part of the doubling
condition on the hypersurface.

Recall that $\tilde w(x)= \tilde c w(rx) = \tilde cv_{x_0,\lambda}(F(rx))$, $F(B^{n-1}(0,2r)\times \{0\} )\subset  \p \Omega_{x_0,\lambda}$,
and the normal derivative of $v_{x_0,\lambda}$ on $\p \Omega_{x_0,\lambda}$  is zero. After using
the area formula as above, we obtain ($\nabla_T$ denotes the tangential gradient):
\begin{align*}
||\nabla \tilde w||_{L^2(B^{n-1}(0,3/4 ))} &=
\tilde c r^{-\frac{n-3}{2}} ||\nabla  w||_{L^2(B^{n-1}(0,3r/4 ))}  \\
&\geq \tilde c r^{-\frac{n-3}{2}} C^{-1} ||\nabla_T v_{x_0,\lambda}||_{L^2(F(B^{n-1}(0,3r/4 )))} \\
&= \tilde c r^{-\frac{n-3}{2}} C^{-1} ||\nabla v_{x_0,\lambda}||_{L^2(F(B^{n-1}(0,3r/4 )))}, \\
|| (\p \tilde w/\p x_n)||_{L^2(B^{n-1}(0,3/4 ))} & =
\tilde c r^{-\frac{n-3}{2}}  || (\p w/\p x_n)||_{L^2(B^{n-1}(0,3r/4 ))} \\
& \leq  \tilde c r^{-\frac{n-3}{2}} C ||\nabla v_{x_0,\lambda}||_{L^2(F(B^{n-1}(0,3r/4 )))}, \\
\end{align*}
and \eqref{H^1-normal} follows. Together with \eqref{H^1-norm}, we hence obtain 
\begin{equation} \label{H^1_proper-norm}
 ||\tilde w||_{H^1(B^{n-1}(0,3/4 ))}  \geq 2^{-C\lambda^5}. 
\end{equation}

Now we will show that
\begin{equation} \label{H^1-L^2}
 ||\tilde w||_{H^1(B^{n-1}(0,3/4 ))} \geq \tilde\epsilon \ \ \Rightarrow \ \  ||\tilde w||_{L^2(B^{n-1}(0,7/8 ))}\geq \tilde\epsilon^3 /C.
\end{equation}
Introduce a cut-off function $\varphi\in C^{\infty}(\R^n)$ such that
\begin{align*}
\varphi &=1 \ \ \ \text{ on } B(0,3/4), \\
\varphi &=0 \ \ \ \text{ on } \R^n \setminus B(0,7/8), \\
|| \varphi||_{C^2(\R^n)} &\leq C.
\end{align*}
Then using Lemma \ref{lemma_trace_r_n-1} for the function $\tilde w'=\tilde w\cdot \varphi$ (extended by $0$ beyond $\R^n \setminus B(0,7r/8)$),
we get that for any $\eta>0$,
\begin{equation} \label{grad_w_l^2}
\begin{aligned}
 ||\nabla \tilde w||_{L^2(B^{n-1}(0,3/4 ))} &\leq ||\nabla \tilde w'||_{L^2(\R^{n-1})} \\
&\leq \eta ||\tilde w'||_{H^2(\R^n)} + \frac{C}{\eta^2}||\tilde w'||_{L^2(\R^{n-1})} \\
&\leq C\eta ||\tilde w||_{H^2(B(0,7/8))} + \frac{C}{\eta^2}||\tilde w||_{L^2(B^{n-1}(0,7/8))} \\
&\leq C\eta  ||\tilde w||_{L^2(B(0,1))} + \frac{C}{\eta^2}|| \tilde w||_{L^2(B^{n-1}(0,7/8))} \\
&\leq C\eta  + \frac{C}{\eta^2}||\tilde w||_{L^2(B^{n-1}(0,7/8))},
\end{aligned}
 \end{equation}
where in the last two lines we used the interior elliptic estimate for strong solutions
$$
||\tilde  w||_{H^2(B(0,7/8))} \leq C||\tilde w||_{L^2(B(0,1))}
$$
(see \cite[Theorem 9.11]{trudinger}) and \eqref{int_2r}.
Choosing $\eta=  \tilde\epsilon/2C $ in \eqref{grad_w_l^2}, where 
$$||\tilde w||_{H^1(B^{n-1}(0,3/4 ))} \geq \tilde\epsilon,$$
we obtain
\begin{equation*}
\begin{aligned}
\tilde\epsilon - ||\tilde w||_{L^2(B^{n-1}(0,3/4))} &\leq ||\nabla \tilde w||_{L^2(B^{n-1}(0,3/4 ))} \\
&\leq  \tilde\epsilon/2  + \frac{C}{ \tilde\epsilon^2}||\tilde w||_{L^2(B^{n-1}(0,7/8))},
\end{aligned}
\end{equation*}
and \eqref{H^1-L^2} follows. Together with \eqref{H^1_proper-norm}, we obtain (with a different constant)
\begin{equation} \label{L^2-norm_bdry}
 ||\tilde w||_{L^2(B^{n-1}(0,7/8 ))}  \geq 2^{-C\lambda^5}. 
\end{equation}
On the other hand, the trace theorem, the interior elliptic estimate \cite[Theorem 9.11]{trudinger}, 
 and  \eqref{int_2r} imply
\begin{equation*}
  ||\tilde w||_{L^2(B^{n-1}(0,7/4 ))}  \leq  C ||\tilde w||_{H^1(B^{n}(0,7/4 ))} 
\leq \tilde C||\tilde w||_{L^2(B(0,2))} = \tilde C,
\end{equation*}
so
\begin{equation*}
  ||\tilde w||_{L^2(B^{n-1}(0,7/8 ))} \geq 2^{-C\lambda^5} ||\tilde w||_{L^2(B^{n-1}(0,7/4 ))},
\end{equation*}
which translates into
\begin{equation*}
  || w||_{L^2(B^{n-1}(0,7r/8 ))} \geq 2^{-C\lambda^5} || w||_{L^2(B^{n-1}(0,7r/4 ))}.
\end{equation*}
Since $r\leq r_0/2(1+\epsilon)$ was arbitrary, we proved \eqref{doubling_bdry_w} for any $r\leq r_1$,
$r_1 = 7 r_0/ (16\cdot 1.1)$.
\end{proof}

In the process of the above proof, we showed that the average integral of $u^2$ over a small ball on the boundary
controls the average integral of $u^2$ over a small solid ball with the same center. Since we will use
this result in the next chapter (independently of the doubling condition), let us formulate it precisely.
\begin{cor} \label{bdry_controls_solid}
 There exist $C,r_0>0$ depending only on $\Omega$ such that for all $r\leq r_0$ and $x_0\in\pO$, there holds
\begin{equation} \label{L^2_bdry_controls_solid}
\int_{B\left(x_0,\frac{r}{\lambda}\right) \cap \pO} u^2 \geq 2^{-C\lambda^5} \frac{\lambda}{r} \int_{B\left(x_0,\frac{2r}{\lambda}\right)\cap \Omega} u^2.
\end{equation}
\end{cor}

\begin{proof}
This follows from \eqref{L^2-norm_bdry}. Looking back at the definition of $\tilde{c}$ by 
\eqref{int_2r},
we have
\begin{equation*}
\begin{aligned}
 \frac{1}{\tilde{c}^2} &= \int_{B(0,2)} w^2(rx) dx = \frac{1}{r^n} \int_{B(0,2r)}v_{x_0,\lambda}^2(F(x)) \\
&\geq \frac{\lambda^n}{r^n} \int_{B\left(x_0,\frac{2r}{(1+\epsilon)\lambda}\right)} v^2 
\geq \frac{\lambda^n}{r^n} \int_{B\left(x_0,\frac{2r}{(1+\epsilon)\lambda}\right) \cap \Omega } u^2.
\end{aligned}
\end{equation*}
We  used a change of variables by the map $F$, whose Jacobian is 1 and for which \eqref{F_subset} holds.
On the other hand, from \eqref{L^2-norm_bdry}, we obtain
\begin{equation*}
\begin{aligned}
\frac{1}{\tilde{c}^2} 2^{-C\lambda^5} &\leq \frac{1}{\tilde{c}^2} \int_{B^{n-1}(0,7/8)} \tilde w^2
= \frac{1}{r^{n-1}} \int_{B^{n-1}(0,7/8 r)} w^2 \\
&= \frac{1}{r^{n-1}} \int_{B^{n-1}(0,7/8 r)} v_{x_0,\lambda}^2(F(x)) \\
&\leq \frac{\lambda^{n-1}}{r^{n-1}} \int_{B\left(x_0, \frac{7(1+\epsilon)}{8} r\right) \cap \pO} v^2
= \frac{\lambda^{n-1}}{r^{n-1}} \int_{B\left(x_0, \frac{7(1+\epsilon)}{8} r\right) \cap \pO} u^2.
\end{aligned}
 \end{equation*}
Here we used a change of variables by the map $F|_{B^{n-1}(0,r)\times\{0\}}$, whose
Jacobian is larger or equal to 1 and for which \eqref{F_supset_bdry} holds.

Combining the above two estimates and using $\epsilon$ such that $(1+\epsilon)^2<8/7$
(this gives us $r_0>0$ depending only on $\Omega$ such that the estimates hold
for $r\leq r_0$), we obtain \eqref{L^2_bdry_controls_solid}.
\end{proof}



\section{Nodal Sets of Steklov Eigenfunctions on Analytic Domains} \label{chap_analytic}

In this section we prove Theorem \ref{thm_steklov_analytic}: 
{
\renewcommand{\thethm}{\ref{thm_steklov_analytic}}
\begin{thm}
 Let $\Omega \subset \R^n$ be an analytic domain. Then there exists a constant $C$ depending
only on $\Omega$ and $n$ such that for any $\lambda>0$ and $u$ 
which is a (classical) solution to \eqref{main} there holds
\begin{equation} \label{nodal_set_analytic}
 \mathcal H^{n-2} (\{x\in\pO: u(x)=0\}) \leq C \lambda^6.
\end{equation}
\end{thm}
}
The proof relies on the doubling condition on the boundary (Theorem \ref{thm_doubling_bdry})
and the approach used in \cite{lin_nodal} for analytic solutions of elliptic equations
on a solid domain.

For the rest of this section, assume the assumptions of Theorem \ref{thm_steklov_analytic}: $\Omega \subset \R^n$
is an analytic domain and $u$ is a solution to \eqref{main} with eigenvalue $\lambda>0$.

\subsection{Analyticity of the eigenfunctions} 

First we will see that $u$ is analytic up to the boundary $\pO$, and get an estimate on its derivatives
which we will need in a complexification argument.

\begin{prop} \label{derivatives_for_complexification}
The eigenfunction $u$ is real-analytic on a neighborhood of $\overline{\Omega}$. 
More precisely,
there exists $r_0 >0$ depending only on $\Omega$ such that
$u$ can be harmonically extended onto an $\frac{r_0}{\lambda}$-neighborhood of $\Omega$.
Moreover, there exists a constant $C$ such that for every $x_0\in\pO$ and $r\leq r_0$,
\begin{equation} \label{derivative_estimate}
  |D^{\alpha} u(x_0)| \leq C \frac{\alpha !}{(r/2ne\lambda)^{|\alpha|}} \left( \dashint_{B(x_0,r/\lambda)} u^2 \right)^{1/2}.
\end{equation}
\end{prop}

\begin{proof}

By the theorem on elliptic iterates of Lions and Magenes (\cite[Theorem VIII.1.2]{lm2}),
$u$ is analytic on a neighborhood of $\overline{\Omega}$.  
To determine the size of this neighborhood, 
for a fixed $x_0\in\pO$, in a similar way as in Section \ref{reflection},
denote
\begin{equation*}
u_{x_0,\lambda}(x) := u(x_0+ x/\lambda), \ \  x\in 
\overline{\Omega_{x_0,\lambda}}, 
\end{equation*}
where $\Omega_{x_0,\lambda} = \{x: x_0+x/\lambda \in\Omega \}$. 
By \eqref{main} 
we have
\begin{equation*}
\begin{aligned}
\Delta u_{x_0,\lambda} &= 0  &\text{ in }  \Omega_{x_0,\lambda},\\
\frac{\p u}{\p \nu} &=u &\text{ on } \partial \Omega_{x_0,\lambda}. 
\end{aligned}
\end{equation*}
It follows from sections VIII.1 and VIII.2 in \cite{lm2} that
$u_{x_0,\lambda}$ can be analytically extended onto $B(0, r_0)$, where $r_0$ depends
only on $\Omega$, not on $x_0$ or $\lambda$
(the domain $\Omega_{x_0,\lambda}$ changes, but only becomes flatter, i.e.~better with increasing $\lambda$).
This means that $u$ can be analytically extended onto $B(x_0,r_0/\lambda)$.
Then $\Delta u$ is also analytic on $B(x_0,r_0/\lambda)$ 
and is equal to zero on the open set $B(x_0,r_0/\lambda) \cap \Omega$, hence it is zero in $B(x_0,r_0/\lambda)$ 
and $u$ is harmonic in $B(x_0,r_0/\lambda)$. 
Hence, from Proposition 1.13 and Remark 1.19 in \cite{lin_skripta}, we obtain
\begin{equation*}
\begin{aligned}
 |D^{\alpha} u(x_0)| &\leq C \frac{\alpha !}{(r/2n e \lambda)^{|\alpha|}} \max_{\overline{B(x_0,r/2 \lambda)}}|u| \\
&\leq C \frac{\alpha !}{(r/2ne\lambda)^{|\alpha|}} \left( \dashint_{B(x_0,r/\lambda)} u^2 \right)^{1/2}.
\end{aligned}
\end{equation*}
\end{proof}


\begin{cor} \label{u_complexification}
 Let $r_0$ be as in Proposition \ref{derivatives_for_complexification}, $x_0\in\pO$. Then
$u$ can be extended into the complex ball $B^{\mathbb{C}^n}(x_0,r_0/6n\lambda)$
such that for any $r\leq r_0$,
\begin{equation} \label{l_infty_complex}
 \sup_{z\in B^{\mathbb{C}^n}(x_0,r/6n\lambda)} |u(z)| \leq C \left( \dashint_{B(x_0,r/\lambda)} u^2 \right)^{1/2}.
\end{equation}
\end{cor}

\begin{proof}
 Write the real Taylor polynomial of $u$ in $x_0$. By the estimate \eqref{derivative_estimate} on its coefficients
$|D^{\alpha} u(x_0)|/ \alpha ! $, it converges for all complex $z\in B^{\mathbb{C}^n}(x_0,r_0/2ne\lambda)$,
and for $z\in B^{\mathbb{C}^n}(x_0,r/6n\lambda)$ we also obtain the estimate \eqref{l_infty_complex}.
\end{proof}

\subsection{Proof of Theorem \ref{thm_steklov_analytic}}

\noindent
{\bf Flattening the boundary and complexification:}

As in the proof of Theorem \ref{thm_doubling_bdry} in Section \ref{sec-doubling-on-bdry},
we cover $\pO$ by finitely many pieces $\Gamma_i\subset \pO$, $i=1,2,\dots, k$, such that 
$k$ depends only on $\Omega$ and each $\Gamma_i$
is a graph of a (this time) analytic function (in some coordinate system). 
We require some additional properties of this cover.
For each 
$\Gamma_i=\Gamma$, after choosing the right coordinate system, we want 
\begin{equation*}
\Gamma = \{ x=(x',x_n)\in \R^n: x'\in\tilde{\Gamma}, x_n=\Phi(x') \}, 
\end{equation*}
where $\tilde{\Gamma} \subset \R^{n-1}$ is a compact set, $\Phi$ is analytic 
on the larger set 
$$\tilde{\Gamma}'=\left\{x\in \R^{n-1}: \dist(x,\tilde{\Gamma})< \frac{2 r_0}{\lambda_1}  \right\},$$
where $r_0>0$ is a constant depending on $\Omega$,
$$\{(x',  \Phi(x')): x'\in \tilde{\Gamma}' \} \subset \pO,$$ 
$\nabla \Phi (x_0) = 0$ for some point $x_0\in \tilde{\Gamma}$,
and $\Phi$ can be extended to a complex analytic function on 
$$\tilde{\Gamma}' \times \left(\frac{-r_0}{\lambda_1},\frac{r_0}{\lambda_1}\right)^{n-1}$$ 
with $\nabla_z \Phi(x_0) = 0$,
\begin{equation} \label{Phi_derivative_bound}
|\nabla_z \Phi(z)| \leq 0.1  \ \ \text{ for } z\in \tilde{\Gamma}' \times \left(\frac{- r_0}{\lambda_1},\frac{ r_0}{\lambda_1}\right)^{n-1}.
\end{equation}
This can be achieved by considering a neighborhood of each $x_0\in \pO$ which can be parametrized
this way and then choosing a finite cover of $\pO$ by these neighborhoods. Denote
$$\tilde{\Gamma}_m'=\left\{x\in \R^{n-1}: \dist(x,\tilde{\Gamma})< \frac{ r_0}{\lambda_1}  \right\}$$
an ``intermediate'' set ($\tilde\Gamma \subset \tilde\Gamma_m' \subset \tilde\Gamma'$).

In a similar way as in the proof of Theorem \ref{thm_doubling_bdry}, denote
\begin{equation*}
w(z'):=u(z',\Phi(z')).  
\end{equation*}
By the definition of $\Phi$ and
Corollary \ref{u_complexification}, this is well-defined if 
\begin{equation*}
z'\in \tilde{\Gamma}' \times \left(\frac{-r_0}{\lambda_1},\frac{r_0}{\lambda_1}\right)^{n-1} \qquad \text{ and } \qquad
\dist((z',\Phi(z')),\pO )\leq \frac{r_0}{6n\lambda} ,
\end{equation*}
where $r_0$ is the minimum of $r_0$'s coming up in Corollary \ref{u_complexification} and in the discussion above
about properties of $\Phi$. If $z'=x'+iy'$, $x'\in \tilde{\Gamma}'$, $|y'|\leq \frac{r_0}{\lambda_1}$, then $(x',\Phi(x'))\in\pO$, and
we have
\begin{equation*}
\begin{aligned}
\dist ((z',\Phi(z')),\pO )\leq |(x'+iy', \Phi (x'+iy')) - (x',\Phi(x'))| \\
\leq |y'| + |y'| \sup_{t\in[0,1]} |\nabla_z \Phi (x'+ity')|
\leq 1.1 |y'| 
\end{aligned}
\end{equation*}
by \eqref{Phi_derivative_bound}. 
Hence, $w$ is well-defined and analytic for
\begin{equation*}
z'=x'+iy',  \ \ x'\in \tilde{\Gamma}',\ \  |y'|\leq \frac{r_0}{6.6n\lambda}, 
\end{equation*}
and 
by Corollary \ref{u_complexification}, for all $r\leq r_0$, $x'\in \tilde{\Gamma}_m'$ we also have
\begin{equation} \label{w_L_infinity}
 \sup_{ B^{\C^{n-1}}\left(x', \frac{r}{ 6.6n\lambda}\right)} |w| \leq C \left( \dashint_{B((x',\Phi(x')),r/\lambda)} u^2 \right)^{1/2}.
\end{equation}
Instead of the bound on the right hand side which involves an integral over $\R^n$-balls, 
using Corollary \ref{bdry_controls_solid} we can bound the left hand side by an integral
over balls on $\pO$. Denote $x_0=(x',\Phi(x'))$.
Although Corollary \ref{bdry_controls_solid} gives a bound only on
an integral over $B(x_0, r/\lambda) \cap \Omega$, in our case $u$ is harmonic
also on $B(x_0, r/\lambda) \cap \Omega^c$ and satisfies the almost identical
boundary condition $\frac{\p u}{\p \nu} = - \lambda u$ on $\p \Omega^c$.
Since all arguments in the proof of Corollary \ref{bdry_controls_solid}
are local and the sign of $\lambda$ never mattered, applying it to $B(x_0, r/\lambda) \cap \Omega^c$ gives us
\begin{equation*}
 \int_{B\left(x_0,\frac{r}{2\lambda}\right) \cap \pO} u^2 \geq 2^{-C\lambda^5} \frac{\lambda}{r} \int_{B\left(x_0,\frac{r}{\lambda}\right)\cap \Omega^c} u^2
\end{equation*}
and together with Corollary \ref{bdry_controls_solid} for regular $\Omega$ and \eqref{w_L_infinity}, we obtain
\begin{equation*}
  \sup_{B^{\C^{n-1}} \left(x', \frac{r}{ 3.3n\lambda}\right)} |w| \leq  2^{C\lambda^5} \left( \left(\frac{\lambda}{r}\right)^{n-1} \int_{B\left(x_0,\frac{r}{\lambda}\right) \cap \pO} u^2 \right)^{1/2}
\end{equation*}
for all  $x'\in \tilde{\Gamma}_m'$, $x_0=(x',\Phi(x'))$, 
and $r\leq r_0$, where $r_0$ depends only on $\Omega$ (it is different than above).
By a change of variables as in the proof of Theorem \ref{thm_doubling_bdry}, using \eqref{Phi_derivative_bound}
for bounding the Jacobian, this 
implies
\begin{equation} \label{w_L_infty_brdy}
  \sup_{B^{\C^{n-1}}\left(x', \frac{r}{ 4n\lambda}\right)} |w| \leq  2^{C\lambda^5} 
\left( \dashint_{B^{n-1}\left(x',\frac{r}{\lambda}\right) } w^2 \right)^{1/2}
\end{equation}
for all $x'\in \tilde\Gamma_m'$ and $r\leq r_0$, where $r_0>0$  depends only on $\Omega$ (different from $r_0$ above).

Also note that the doubling condition, Theorem \ref{thm_doubling_bdry}, translates into the following for $w$:
\begin{equation} \label{doubling_bdry_w_new}
  \int_{B^{n-1}(x',\frac{r}{\lambda}) }  w^2 \leq  2^{C\lambda^5} \int_{ B^{n-1}(x',\frac{r}{2\lambda}) }  w^2
\end{equation}
for all $x'\in \tilde\Gamma_m'$ and $r\leq r_0$, where $r_0>0$  depends only on $\Omega$ (different from $r_0$ above).
We actually proved this in the process of proving Theorem \ref{thm_doubling_bdry} as \eqref{doubling_bdry_w}
(the $w|_{B^{n-1}(0,r_0) \times\{0\}} $ in there was defined exactly as our $w$ shifted to $0$ and scaled by $\lambda$).

\medskip

By the change of variable as in the proof of Theorem \ref{thm_doubling_bdry}, using \eqref{Phi_derivative_bound},
to prove \eqref{nodal_set_analytic}, it is enough to prove
\begin{equation} \label{nodal_set_analytic_w}
 \mathcal H^{n-2} (\{x'\in \tilde{\Gamma}: w(x')=0\}) \leq C \lambda^6.
\end{equation}
To estimate the nodal set, we use an estimate on the number of zero points
for analytic functions. 

\begin{lemma} \label{lemma_complex_nodal}
 Suppose $f: B_1 \subset \mathbb C \to \mathbb C$ is analytic with
\begin{equation*}
 |f(0)|=1 \ \ \text{ and } \ \ \sup_{B_1}|f|\leq 2^N,
\end{equation*}
for some positive constant $N$. Then for any $r\in (0,1)$ there holds
\begin{equation*}
 \# \{z\in B_r: f(z)=0 \} \leq cN,
\end{equation*}
where $c$ is a positive constant depending only on $r$. For $r=1/2$, we have
\begin{equation*}
 \# \{z\in B_{1/2}: f(z)=0 \} \leq N.
\end{equation*}
\end{lemma}

\begin{proof}
A similar version of this lemma was first proved in \cite{donnelly_fefferman}. This version can be found in
\cite{book} as Lemma 2.3.2.
\end{proof}

\begin{proof}[Proof of Theorem \ref{thm_steklov_analytic}]
It remains to prove \eqref{nodal_set_analytic_w}.
Take $r_0$ such that both \eqref{w_L_infty_brdy} and \eqref{doubling_bdry_w_new} hold for all $r\leq r_0$ and $x'\in\tilde\Gamma_m'$.
Choose a cover of $\tilde\Gamma$ by balls of radii $R=\frac{r_0}{(16n+1)\lambda}$. We need $C \lambda^{n-1}$ balls in the cover,
where $C$ depends only on $\Omega$. Take one of them and denote it $B(p,R)$, i.e.~the center is $p\in \tilde\Gamma$.
Choose $x_p\in B(p,R) \subset \tilde\Gamma_m'$ such that
\begin{equation} \label{x_p_large}
 |w(x_p)| \geq \left(\dashint_{B(p,R)}|w|^2 \right)^{1/2}.
\end{equation}
Then $B(p,R) \subset B(x_p,2R)$, so the $C\lambda^{n-1}$ balls $B(x_p,2R)$ cover $\tilde\Gamma$.
Next, 
$$B(x_p,16nR)\subset B(p,(16n+1)R),$$ 
and since $4R<r_0/4n\lambda$, by \eqref{w_L_infty_brdy} we have
\begin{equation*}
\begin{aligned}
 \sup_{B^{\C^{n-1}}(x_p, 4R)} |w| &\leq 2^{C\lambda^5} \left( \dashint_{B^{n-1}\left(x_p,{16nR}\right) } w^2 \right)^{1/2} \\
&\leq 2^{C\lambda^5} \left( \dashint_{B^{n-1}\left(p,(16n+1)R\right) } w^2 \right)^{1/2}.
\end{aligned}
\end{equation*}
Now use the doubling condition \eqref{doubling_bdry_w_new} with center $p$, $\lceil \log_2(16n+1) \rceil$-times,
and obtain
\begin{equation} \label{sup_estimate_w}
 \sup_{B^{\C^{n-1}}(x_p, 4R)} |w| \leq 2^{C\lambda^5} \left( \dashint_{B^{n-1}\left(p,R\right) } w^2 \right)^{1/2}
\leq  2^{C\lambda^5} |w(x_p)|
\end{equation}
(we used \eqref{x_p_large} in the last inequality). For each direction $\omega \in \R^{n-1}$, $|\omega|=1$,
consider the function $f_{\omega}(z)= w(x_p + 4R z \omega)$ of one complex variable $z$ defined on $B^{\C}(0,1)$.
From \eqref{sup_estimate_w} and Lemma \ref{lemma_complex_nodal}, we obtain 
\begin{equation*}
\begin{aligned}
\# \{x\in B^{n-1}(x_p,2R) : x-x_p || \omega,  w(x)=0 \} & \leq
\# \{z\in B_{1/2}: f_{\omega}(z)=0 \} \\
&= N(\omega)\leq C \lambda^5. 
\end{aligned}
\end{equation*}
By the integral geometry estimate \cite[3.2.22]{federer}, we obtain
\begin{equation*}
\mathcal H^{n-2} \{x\in B^{n-1}(x_p,2R): w(x)=0  \} \leq c(n)R^{n-2}\int_{\mathbb S^{n-2}} N(\omega) d\omega  \leq C  \frac{1}{\lambda^{n-2}} \lambda^5.
\end{equation*}
Since there were $C \lambda^{n-1}$ balls $B^{n-1}(x_p,2R)$ which cover $\tilde\Gamma$, summing up these estimates gives \eqref{nodal_set_analytic_w}.
\end{proof}

\bibliographystyle{spmpsci}

\end{document}